\documentclass[11pt]{article}

\usepackage{amssymb}
\usepackage{amsthm}
\usepackage{amsmath}
\usepackage[dvipdfmx]{graphicx}
\usepackage[
driverfallback=dvipdfm,
bookmarks=true,
bookmarksnumbered=true, 
bookmarkstype=toc,
setpagesize=false,
pdftitle={Dual spectral projection gradient method
for a log-determinant semidefinite problem},
pdfauthor={Makoto Yamashita},
pdfkeywords={spectral projection gradient method,
non-monotone gradient method
}
]{hyperref}

\usepackage{fancybox,color}
\definecolor{lred}{rgb}{1,0.8,0.8}
\definecolor{lblue}{rgb}{0.8,0.8,1}
\definecolor{dred}{rgb}{0.6,0,0}
\definecolor{dblue}{rgb}{0,0,0.5}
\definecolor{violet}{rgb}{0.5804,0.0000,0.8275}
\definecolor{purple}{rgb}{0.2400,0.5700,0.2500}
\def\@themcountersep{}
\newtheorem{THEO}{Theorem}[section]
\newtheorem{ALGO}[THEO]{Algorithm}

\newtheorem{LEMM}[THEO]{Lemma}

\newtheorem{PROP}[THEO]{Proposition}
\newtheorem{REMA}[THEO]{Remark}

\def\bold#1{\mbox{\boldmath $#1$}}

\def\0{\mbox{\bf 0}}
\def\1{\mbox{\bf 1}}
\def\2{\mbox{\bf 2}}
\def\3{\mbox{\bf 3}}
\def\4{\mbox{\bf 4}}
\def\5{\mbox{\bf 5}}
\def\6{\mbox{\bf 6}}
\def\7{\mbox{\bf 7}}
\def\8{\mbox{\bf 8}}
\def\9{\mbox{\bf 9}}

\def\b{\mbox{\boldmath $b$}}

\newdimen\zhige \zhige=0pt
\def\chige#1{{\setbox\zhige\hbox{#1}\ifdim\ht\zhige=1ex\accent24 #1%
  \else\ooalign{\unhbox\zhige\crcr\hidewidth\char24\hidewidth}\fi}}

\def\s{\mbox{\boldmath $s$}}

\def\x{\mbox{\boldmath $x$}}
\def\y{\mbox{\boldmath $y$}}

\def\A{\mbox{\boldmath $A$}}

\def\C{\mbox{\boldmath $C$}}

\def\E{\mbox{\boldmath $E$}}

\def\I{\mbox{\boldmath $I$}}

\def\L{\mbox{\boldmath $L$}}

\def\O{\mbox{\boldmath $O$}}
\def\P{\mbox{\boldmath $P$}}

\def\S{\mbox{\boldmath $S$}}

\def\U{\mbox{\boldmath $U$}}
\def\V{\mbox{\boldmath $V$}}
\def\W{\mbox{\boldmath $W$}}
\def\X{\mbox{\boldmath $X$}}
\def\Y{\mbox{\boldmath $Y$}}
\def\Z{\mbox{\boldmath $Z$}}
\def\AC{\mbox{$\cal A$}}

\def\DC{\mbox{$\cal D$}}

\def\FC{\mbox{$\cal F$}}

\def\LC{\mbox{$\cal L$}}

\def\NC{\mbox{$\cal N$}}

\def\PC{\mbox{$\cal P$}}

\def\WC{\mbox{$\cal W$}}

\def\Real{\mbox{$\mathbb{R}$}}

\def\SMAT{\mbox{$\mathbb{S}$}}

\def\sX{\mbox{\scriptsize $\X$}}
\def\sZ{\mbox{\scriptsize $\Z$}}
\def\sW{\mbox{\scriptsize $\W$}}
\def\sWC{\mbox{\scriptsize $\WC$}}
\usepackage[normalem]{ulem}
\ifx11 

\fi

\title{A dual spectral projected gradient method
for  log-determinant semidefinite problems}

\author{ 
Takashi Nakagaki\thanks{Department of  Mathematical and Computing Science,
        Tokyo Institute of Technology, 2-12-1-W8-41 Oh-Okayama, Meguro-ku, Tokyo 152-8552, Japan. }
\and 
\normalsize
        Mituhiro Fukuda\thanks{Department of  Mathematical and Computing Science,
        Tokyo Institute of Technology, 2-12-1-W8-41 Oh-Okayama, Meguro-ku, Tokyo 152-8552, Japan
         ({\tt mituhiro@is.titech.ac.jp}). 
	The research was partially supported by JSPS KAKENHI (Grant number: 26330024), 
        and by the Research Institute for Mathematical Sciences, a 
	Joint Usage/Research Center located in Kyoto University. } 
\and 
\normalsize 
	Sunyoung Kim\thanks{Department of Mathematics, Ewha W. University, 52 Ewhayeodae-gil, Sudaemoon-gu, 
	Seoul 	03760, Korea  ({\tt skim@ewha.ac.kr}). The research was supported
        by  NRF 2017-R1A2B2005119.}
 \and 	
\normalsize
        Makoto Yamashita\thanks{Department of  Mathematical and Computing Science,
            Tokyo Institute of Technology, 2-12-1-W8-29 Oh-Okayama, Meguro-ku, Tokyo 152-8552, Japan        
             ({\tt makoto.yamashita@is.titech.ac.jp}).
	This research was partially supported by JSPS KAKENHI (Grant number: 18K11176).
 	}	
}
\date{November, 2018}

\begin{document}

\maketitle

\begin{abstract} \noindent
We extend the result on the spectral projected gradient method by Birgin {\it et al.} in 2000 to  a log-determinant semidefinite problem (SDP) with linear constraints
	and propose a spectral projected gradient method for the dual problem. Our method is
	based on alternate projections on the intersection of two convex sets, which first projects   
	onto the box constraints and then onto a set defined by a linear matrix inequality.
By exploiting structures of the two projections, we show the same convergence properties can be obtained for the proposed method as  Birgin's method where	
the exact orthogonal projection onto the intersection of two convex sets is performed.
Using the convergence properties, we prove that the proposed algorithm attains the optimal value or
terminates in a finite number of iterations.
The efficiency of the proposed method is illustrated with the numerical results   
on  randomly generated synthetic/deterministic data and gene expression data, in comparison with
other methods including
 the inexact primal-dual path-following
interior-point method, 
the adaptive spectral projected gradient method, and 
the adaptive
Nesterov's smooth method. 
For the gene expression data, our results are compared with  the quadratic approximation for sparse inverse
covariance estimation method. 
We show that our method outperforms the other methods in obtaining a better optimal value fast.

\end{abstract}

 
\vspace{0.5cm}


\section{Introduction}\label{sec:introduction}

We consider a convex semidefinite program with linear constraints of the form:
\begin{eqnarray*}
(\PC) & & \begin{array}{rcl}
  \min &:& f(\X) := \text{Tr}(\C \X) - \mu \log \det \X + \text{Tr}(\bold{\rho} |\X|) \\
        \mbox{s.t.}  & : & \AC(\X) = \b, \X \succ \O,
      \end{array}
\end{eqnarray*}
where  $\C$, $\X$ and  $\bold{\rho}$ are $n\times n$ symmetric matrices  $\SMAT^n$, the elements of $\bold{\rho} \in \SMAT^n$ are nonnegative,
$\text{Tr}$ denotes the trace of a matrix,
$|\X| \in \SMAT^n$  the matrix obtained by taking the absolute value of every element $X_{ij} \ (1 \leq i,j \leq n)$
of $\X$,   $\X \succ \O$ means that $\X$ is positive definite, and $\AC$  a linear map of $\SMAT^n \to \Real^m$.
 In $(\PC)$,
$\C, \bold{\rho} \in \SMAT^n,
\mu > 0, \b \in \Real^m$, 
and the linear map $\AC$ 
given by
$\AC(\X) = (\text{Tr}(\A_1 \X), \ldots, \text{Tr}(\A_m \X))^T$, where $\A_1,\ldots,\A_m \in \SMAT^n$, are input data.

Problem $(\PC)$ frequently 
appears in statistical models such as sparse covariance selection or Gaussian graphical models.
In particular,
the sparse covariance selection model \cite{DEMPSTER72} or its graphical
interpretation known as Gaussian Graphical Model (GGM) \cite{LAURITZEN96} are
special cases of $(\PC)$ for $\bold{\rho}=\O$ and  
linear constraints taking the form $X_{ij}=0$ for $(i,j)\in\Omega
\subseteq\{(i,j) \ | \ 1\leq i<j\leq n\}$.

Many approximate solution methods for solving variants of $(\PC)$ have been proposed over the years.
The methods mentioned below are mainly from  recent computational developments. 
The adaptive spectral gradient (ASPG) method and the adaptive Nesterov's
smooth (ANS) method proposed by Lu~\cite{LU10} are one of the earlier
methods which can handle large-scale problems. 
Ueno and Tsuchiya~\cite{UENO09} proposed a Newton method  by localized approximation of the relevant data. 
Wang {\it et al.} \cite{WANG10} considered a primal proximal point algorithm which solves
semismooth subproblems by the Newton-CG iterates.
Employing  the inexact primal-dual
path-following interior-point method,
  Li and Toh in~\cite{LI10} demonstrated that the computational efficiency could be increased, despite the known inefficiency of
interior-point methods for solving large-sized problems. 
Yuan~\cite{YUAN2012} also proposed an improved Alternating Direction Method (ADM) to solve the sparse covariance problem by introducing an ADM-oriented reformulation. 
For a more general structured models/problems, Yang {\it et al.}~\cite{YANG13}
enhanced the method in~\cite{WANG10} to handle block structured sparsity, 
employing an inexact generalized Newton method to solve the dual semismooth subproblem. They
demonstrated that regularization 
using $\|\cdot\|_2$ or $\|\cdot\|_{\infty}$ norms instead of $\|\cdot\|_1$ in
$(\PC)$ are more suitable for the structured models/problems.
Wang~\cite{WANG16} first generated an initial point using the proximal augmented Lagrangian method, then applied the Newton-CG augmented Lagrangian method to problems with an additional convex quadratic term in $(\PC)$.
Li and Xiao \cite{LI18} employed the symmetric Gauss-Seidel-type ADMM in the 
same framework of~\cite{WANG10}. 
A more recent work by Zhang {\it et al.}~\cite{ZHANG18} shows that 
 $(\PC)$ with simple constraints as 
$X_{ij}=0$ for $(i,j)\in\Omega$ can be converted into a more computationally tractable
 problem  for large values of $\bold{\rho}$.
Among the methods mentioned here, only the methods discussed in 
\cite{WANG10,YANG13,WANG16} can handle problems as general as $(\PC)$.

We propose a dual-type spectral projected gradient (SPG)
method to obtain the optimal value of
$(\PC)$. More precisely, 
 an efficient algorithm is designed for the dual problem with $g:$ $\Real^m \times \SMAT^n \to \Real$: 
 \begin{eqnarray*}
(\DC) & & \begin{array}{rcl}
  \max &:& g(\y, \W) := \b^T \y + \mu \log \det (\C + \W - \AC^T (\y))
  + n \mu - n \mu \log \mu \\
        \mbox{s.t.}  & : & |\W| \le \bold{\rho},
        \C + \W - \AC^T(\y) \succ \O,  
      \end{array}
\end{eqnarray*}
under
the  three assumptions: 
(i) $\AC$ is surjective, that is, the set of $\A_1,\ldots,\A_m$ is linearly
independent; 
(ii) The problem $(\PC)$ has an interior feasible point, {\it{i.e.}}, 
there exists $\X \succ \O$ such that $\AC(\X) = \b$; 
(iii) A feasible point for $(\DC)$ is given or can be easily computed. 
{\it{i.e.}}, 
there exists
$\y \in\Real^m$ and $\W \in\SMAT^n$
such that $| \W |\leq \bold{\rho}$ and $\C+ \W +\AC^T(\y)\succ \O$.
These assumptions are not strong as many applications 
satisfy these assumptions with slight modifications.

Our approach for solving $(\DC)$ by a projected gradient method
is not the first one.  A dual approach was examined  in~\cite{DUCHI08}, however, 
their algorithm which employs the classical gradient  projection method 
cannot handle  linear constraints.  

The spectral projection gradient (SPG) method
by Birgin {\it et al.}~\cite{BIRGIN00}, which is slightly modified 
 in our method, 
minimizes a smooth objective function over a closed convex set.
Each iteration of the SPG requires (a) projection(s) onto the feasible closed convex set and performs
a non-monotone line search for the Barzilai-Borwein step size~\cite{BARZILAI88}.
An important advantage of the SPG method is that it requires only 
the information of  function values and  first-order derivatives, therefore, the  computational 
cost of each iteration is much less than methods 
which employ second-order derivatives such as interior-point methods.
The ASPG method~\cite{LU10} described above
repeatedly applies the SPG method by decreasing
$\bold{\rho}$ adaptively, but the ASPG method was designed for the
only specific constraint $X_{ij} = 0 \ \mbox{for} \  (i,j) \in \Omega$.
We extend these results to directly handle 
a more general linear constraint $\AC(\X) = \b$.

Our proposed algorithm called Dual SPG, which is a dual-type SPG,  adapts the SPG methods of 
\cite{BIRGIN00} to
$(\DC)$. A crucial difference between our method and the original method is that
 the Dual SPG first performs  an orthogonal projection onto the 
box constraints and subsequently onto the set defined by an LMI, while the original method
computes the exact
orthogonal projection of the search direction over the intersection of the 
two convex sets.
The projection onto the intersection of the two sets requires some iterative methods, 
which frequently causes  some numerical difficulties. Moreover,
the projection by an iterative method is usually inexact, resulting in the search direction that may not be an ascent direction.  
We note that  an ascent direction is necessary for the convergence analysis as shown in Lemma  \ref{lemm:level} in Section 3.
On the other hand, the projections onto the box constraints and the LMI constraints can be exactly computed 
 within numerical errors.

The convergence analysis for the Dual SPG (Algorithm~\ref{algo:dspg}) presented in Section \ref{section:convergence}
shows that 
such approximate
orthogonal projections do not affect convergence, in fact, the convergence properties of the original SPG also hold
 for the Dual SPG. For instance, 
stopping criteria based on the fixed point of the projection (Lemma~\ref{lemm:stop})
and other properties described in the beginning of Section~\ref{section:convergence} can be proved for the Dual SPG.
The properties are used to  finally
prove that the algorithm either terminates in a finite number of iterations or
successfully attains the optimal value.

We should emphasize that the proof for the original SPG developed in \cite{BIRGIN00}
cannot be applied to the  Dual SPG proposed here. 
As the Dual SPG utilizes the two different projections instead of the orthogonal projection onto the feasible region
in the original SPG, 
a new proof  is necessary, in particular, for Lemma~\ref{lemm:stop} where the properties of  the two projections
are exploited.
We  also use the duality theorem
to prove the convergence of a sub-sequence (Lemma~\ref{lemm:liminf-obj})
since the Dual SPG solves the dual problem.
 Lemma~\ref{lemm:liminf-obj} cannot be obtained by simply applying  the proof in \cite{BIRGIN00}.

The implementation of Algorithm~\ref{algo:dspg}, called DSPG in this paper, were run on three
classes of problems: Randomly generated synthetic data (Section~\ref{subsec:random}), deterministic synthetic data
(Section~\ref{subsec:deterministic}), and gene expression data (Section~\ref{subsec:realdata}; with no constraints) from the literature.
Comparison of the DSPG against high-performance code such as ASPG~\cite{LU10}, ANS~\cite{LU10}, 
 and IIPM~\cite{LI10} shows that our code can be superior or at least competitive with them
in terms of computational time when high accuracy is required.
In particular, against QUIC~\cite{HSIEH14}, the DSPG can be faster for denser instances.

This paper is organized as follows:
We proposed our method DSPG in Section~2.  Section 3 is mainly devoted to the convergence of the proposed method. 
Section 4 presents computational results of the proposed method in comparison with other methods.
For the gene expression data, our results are compared with QUIC.
We finally conclude in Section 5.

\subsection{Notation}\label{sec:notation}

We use $||\y|| := \sqrt{\y^T \y}$ for $\y \in \Real^m$
and $||\W|| := \sqrt{\W \bullet \W}$ for $\W \in \SMAT^n$ where
$\W \bullet \V = \text{Tr}(\W\V) = \sum_{i=1}^n \sum_{j=1}^n W_{ij} V_{ij}$ for $\V \in \SMAT^n$,
as the norm of vectors and matrices, respectively.
We extend the inner-product to the space of $\Real^m \times \SMAT^n$
by
$(\y_1, \W_1) \bullet (\y_2, \W_2) :=
\y_1^T  \y_2 + \W_1 \bullet \W_2$
for 
$(\y_1, \W_1), (\y_2, \W_2) \in \Real^m \times \SMAT^n$.
The norm of linear maps is  defined by
$||\AC|| := \max_{||\X||=1} ||\AC(\X)||$.

The superscript of $T$ indicates the transpose of vectors or matrices, or
the adjoint of linear operators. For example, the
adjoint of $\AC$ is denoted by $\AC^T : \Real^m \to \SMAT^n$.
The notation $\X \succeq \Y (\X \succ \Y)$ stands for $\X - \Y$
being a positive semidefinite matrix (a positive definite matrix, respectively).
We also use $\X \ge \Y$ to describe that $\X-\Y$ is a non-negative matrix,
that is, $X_{ij} \ge Y_{ij}$ for all $i,j = 1,\ldots, n$.

The induced norm for $\Real^m \times \SMAT^n$ is given by
$||(\y, \W)|| := \sqrt{(\y, \W) \bullet (\y, \W)}$.
To evaluate the accuracy of the solution,
we also use an element-wise infinity norm defined by
\[ ||(\y, \W)||_{\infty} :=
\max\{\max_{i=1,\ldots,m} |y_i|, \max_{i,j=1,\ldots,n} |W_{ij}|\}. \]

For a matrix $\W \in \SMAT^n$,
$[\W]_{\le \bold{\rho}}$ is the matrix
whose
$(i,j)$th element is 
$\min\{ \max\{ W_{ij}, -\rho_{ij} \}, \rho_{ij}\}$.
The set of such matrices is denoted by
$\WC := \{[\W]_{\le \bold{\rho}} : \W \in \SMAT^n\}$.
In addition, $\P_S$ denotes the projection onto a closed convex set
$S$;
\begin{eqnarray*}
\P_S (\x) = \mbox{arg}\min_{\y \in S} || \y - \x||.
\end{eqnarray*}

We denote an optimal solution of $(\PC)$ and $(\DC)$ by $\X^*$ and 
$(\y^*, \W^*)$, respectively.
For simplicity, we use $\X(\y,\W) := \mu (\C + \W - \AC^T(\y))^{-1}$. 
The gradient of $g$ is a map of $\Real^m \times \SMAT^n \to \Real^m \times \SMAT^n$ given by
\begin{eqnarray*}
\nabla g(\y,\W) &:=& (\nabla_{\y} g(\y, \W),  \nabla_{\W} g(\y, \W)) \\
&=&(\b - \mu \AC((\C + \W - \AC^T(\y))^{-1}), \mu (\C + \W - \AC^T(\y))^{-1}) \\
&=&(\b - \AC(\X(\y,\W)), \X(\y,\W))
\end{eqnarray*}
We use $\FC$ and $\FC^*$ to denote the feasible set and 
the set of optimal solutions of $(\DC)$, respectively;
\begin{eqnarray*}
\FC &:=& \{(\y, \W) \in \Real^m \times \SMAT^n :
\W \in \WC, \C + \W - \AC^T(\y) \succ \O \}\\
\FC^* &:=& \{(\y^*, \W^*) \in \Real^m \times \SMAT^n :  g(\y^*, \W^*) \ge g(\y, \W)
  \ \mbox{for} \ (\y, \W) \in \FC \}.
\end{eqnarray*}
Finally,  $f^*$ and $g^*$ are used to denote the optimal values of
$(\PC)$ and $(\DC)$, respectively.

\section{Spectral Projected Gradient Method for the Dual Problem}\label{Dual-SPG}

To propose a numerically efficient method, we  focus on the fact that
the feasible region of $(\DC)$ is the intersection of  two convex sets: 
$\FC = \widehat{\WC} \cap \widehat{\FC}$ where
\begin{eqnarray*}
\widehat{\WC} &:=&  \Real^m \times \WC \\
\widehat{\FC} &:=& \{(\y, \W) \in \Real^m \times \SMAT^n :
\C + \W - \AC^T(\y) \succ \O \}.
\end{eqnarray*}
Although the projection onto this 
 intersection requires elaborated computation, 
the projection onto the first set can be 
simply obtained by
\begin{eqnarray}
\P_{\widehat{\WC}} (\y, \W) = (\y, [\W]_{\le \bold{\rho}}).
\label{eq:PW}
\end{eqnarray}
Next, we  consider the second set $\widehat{\FC}$.
If the $k$th iterate $(\y^k, \W^k)$  satisfies
$\C + \W^k - \AC^T (\y^k) \succ \O$ and the direction toward the next iterate
$(\y^{k+1}, \W^{k+1})$ is given by $(\Delta \y^k, \Delta \W^k)$, then
the step length $\lambda$ can be computed such that
$(\y^{k+1}, \W^{k+1}) :=  (\y^k, \W^k) + \lambda (\Delta \y^k, \Delta \W^k)$ satisfies 
$\C + \W^{k+1} - \AC^T (\y^{k+1}) \succ \O$ using a similar procedure to
interior-point methods.
(See \ref{eq:overlabmda} below.) By the assumption (iii), 
we can start from some initial point $(\y^0, \W^0) \in \FC = \widehat{\WC} \cap \widehat{\FC}$ and 
it is easy to keep all the iterations inside the intersection
$\FC$. 

We 
now propose Algorithm~\ref{algo:dspg}
for solving the dual problem $(\DC)$.
The notation $\X^k := \X(\y^k, \W^k) = \mu (\C + \W^k +\AC^T(\y^k))^{-1}$ is used. 
\begin{ALGO} \label{algo:dspg}
(Dual Spectral Projected Gradient Method) \rm
\begin{itemize}
\item[Step 0:] 
          Set parameters
          $\epsilon \ge 0, \ \gamma \in (0,1), \ \tau \in (0,1),\ 
          0 < \sigma_1 < \sigma_2 < 1,\ 
          0 < \alpha_{\min} < \alpha_{\max} < \infty$ and
          an integer parameter $M \ge 1$.
          Take the initial point $(\y^0, \W^0) \in \FC$ and
          an initial projection length
          $\alpha^0 \in [\alpha_{\min}, \alpha_{\max}]$.
          Set an iteration number $k:=0$.
\item[Step 1:]
          Compute a search direction (a projected gradient
          direction) for the stopping criterion
          \begin{eqnarray}
          (\Delta \y^k_{(1)}, \Delta \W^k_{(1)}) &:=&
        \P_{\widehat{\WC}}
          ((\y^k, \W^k) + \nabla g(\y^k, \W^k))
          - (\y^k, \W^k) \nonumber \\
          &=& (\b - \AC(\X^k),
          [\W^k + \X^k]_{\le \bold{\rho}} - \W^k).
          \end{eqnarray}
          If $||(\Delta \y^k_{(1)}, \Delta \W^k_{(1)})||_{\infty} \le \epsilon$,
          stop and output $(\y^k, \W^k)$ as the approximate solution.
\item[Step 2:]
Compute a search direction (a projected gradient
          direction) 
          \begin{eqnarray}
          (\Delta \y^k, \Delta \W^k) &:=&
        \P_{\widehat{\WC}}
          ((\y^k, \W^k) + \alpha^k \nabla g(\y^k, \W^k))
          - (\y^k, \W^k) \nonumber \\
          &=& (\alpha^k(\b - \AC(\X^k)),
          [\W^k + \alpha^k \X^k]_{\le \bold{\rho}} - \W^k).
          \label{eq:search-direction}
          \end{eqnarray}
\item[Step 3:]
          Apply the Cholesky factorization to obtain a lower triangular matrix $\L$
        such that 
          $\C + \W^k - \AC^T (\y^k) = \L \L^T$.
          Let $\theta$ be the minimum eigenvalue of
          $\L^{-1} (\Delta \W^k - \AC^T (\Delta \y^k)) \L^{-T}$.
          Then,  compute
          \begin{eqnarray}
          \overline{\lambda}^k :=
              \left\{\begin{array}{ll}
              1 & (\theta \ge 0) \\
              \min\left\{1, -\frac{1}{\theta} \times \tau \right\}
          & (\theta < 0) \\
                    \end{array}\right. \label{eq:overlabmda}
          \end{eqnarray}
          and set $\lambda_1^k := \overline{\lambda}^k$.
          Set an internal iteration number $j:=1$.
          \begin{itemize}
          \item[Step 3a:]
        Set $(\y_+, \W_+) := (\y^k, \W^k)
                + \lambda_j^k (\Delta \y^k,\Delta \W^k)$.
          \item[Step 3b:]
        If \begin{eqnarray}
          g(\y_+, \W_+)
                \ge \min_{0 \le h \le \min\{k, M-1\}}
                g(\y^{k-h}, \W^{k-h})
                + \gamma \lambda_j^k
                \nabla g(\y^k, \W^k)
                \bullet (\Delta \y^k, \Delta \W^k)
                \label{eq:g-condition}
                \end{eqnarray}
                is satisfied,
                then go to Step 4.
                Otherwise, choose $\lambda_{j+1}^k \in
                [\sigma_1 \lambda_j^k, \sigma_2 \lambda_j^k]$,
                and set $j:= j+1$, and 
                return to Step 3a.
        \end{itemize}
\item[Step 4:]
          Set $\lambda^k := \lambda_j^k$,
          $(\y^{k+1}, \W^{k+1}) := (\y^k, \W^k)
          + \lambda^k (\Delta \y^k, \Delta \W^k)$, 
          $(\s_1, \S_1) := (\y^{k+1}, \W^{k+1}) - (\y^k, \W^k)$
          and
          $(\s_2, \S_2) :=
          \nabla g(\y^{k+1}, \W^{k+1})
          - \nabla g(\y^k, \W^k)$.
          Let $b^k := (\s_1, \S_1) \bullet (\s_2, \S_2)$.
          If $b^k \ge 0$, set
          $\alpha^{k+1} := \alpha_{\max}$.
          Otherwise, let $a^k := (\s_1, \S_1) \bullet (\s_1, \S_1)$ and set
          $\alpha^{k+1} := \min \{\alpha_{\max},
          \max\{\alpha_{\min}, -a^k / b^k\}\}$.
\item[Step 5:]
          Increase the iteration counter $k:= k+1$ and return to Step 1.
\end{itemize}
\end{ALGO}

The projection length $\alpha^{k+1} \in [\alpha_{\min}, \alpha_{\max}]$ 
in Step 4 is based on the Barzilai-Borwein step~\cite{BARZILAI88}.
As investigated in \cite{HAGER08, TAVAKOLI12}, this step
has several advantages. For example,
a linear convergence can be proven for unconstrained optimization
problems without employing line search techniques on the conditions that
its initial point is close to a local minimum and
the Hessian matrix of the objective function is positive definite.

\section{Convergence Analysis}\label{section:convergence}

We prove in Theorem~\ref{theo:limit-obj},  one of our main contributions, that
Algorithm~\ref{algo:dspg} with $\epsilon = 0$
generates a point of $\FC^*$ in a finite number of
iterations or it generates a sequence $\{(\y^k, \W^k)\} \subset \FC $
that attains $\lim_{k \to \infty} g(\y^k, \W^k) = g^*$.

For the proof of Theorem~\ref{theo:limit-obj}, we present lemmas: Lemma \ref{lemm:level}
shows that the sequences $\{(\y^k, \W^k)\}$ by Algorithm~\ref{algo:dspg} remain in a level set of $g$ for each $k$.
Lemma \ref{lemm:bound} discusses on the boundedness of the level set,
Lemma~\ref{lemm:optset} on the uniqueness of 
the optimal solution in $(\PC)$, 
Lemma  \ref{lemm:stop} on the validity of the stopping criteria in Algorithm~\ref{algo:dspg}, 
Lemma \ref{lemm:direction1} on the bounds for the search direction $(\Delta \y^k, \Delta \W^k)$.
Lemmas~\ref{lemm:low-min} and \ref{lemm:liminf-obj}, which use Lemma~\ref{lemm:matrix-norm} in their proofs,
 show that Algorithm~\ref{algo:dspg} does not terminate before computing an approximate solution. 
 Lemma \ref{lemm:low-min} provides a lower bound for the step length $\lambda^k$ of Algorithm~\ref{algo:dspg}.
Lemmas \ref{lemm:liminf-delta}  and \ref{lemm:liminf-obj}, which uses Lemma~\ref{lemm:short},  discuss
the termination of
Algorithm~\ref{algo:dspg} with $\epsilon = 0$  in a finite number of iterations attaining the
optimal value $g^*$ or Algorithm~\ref{algo:dspg} attains 
$\liminf_{k \to \infty} g(\y^k, \W^k) = g^*$.

In  the proof of Theorem~\ref{theo:limit-obj}, 
the properties of projection will be repeatedly used. The representative properties
are summarized  in Proposition 2.1 of \cite{HAGER08}.
We list some of the properties related to this paper in the following and 
their proofs can also be found in \cite{HAGER08} and the references therein.
\begin{PROP}\label{prop:hager} \rm (\cite{HAGER08})
For a convex set $S \subset \Real^n$
and a function $f : \Real^n \to \Real$,
\begin{enumerate}
  \item[(P1)] $(\x - \P_S (\x))^T (\y - \P_S(\x)) \le 0
        \quad \mbox{for} \quad
        \forall \x \in \Real^n, \ \forall \y \in S.$
  \item[(P2)] $(\P_S(\x) - \P_S(\y))^T (\x - \y)
          \ge ||(\P_S(\x) - \P_S(\y)||^2 
        \quad \mbox{for} \quad
        \forall \x, \forall \y \in \Real^n.$
  \item[(P3)] $||\P_S(\x) - \P_S(\y)|| 
          \le ||\x - \y||
        \quad \mbox{for} \quad
        \forall \x, \forall \y \in \Real^n.$
  \item[(P4)] $||\P_S(\x - \alpha \nabla f(\x)) - \x ||$
          is non-decreasing in $\alpha > 0$ for
          $\forall \x \in S$.
  \item[(P5)] $||\P_S(\x - \alpha \nabla f(\x)) - \x || / \alpha$
          is non-increasing in $\alpha > 0$ for
          $\forall \x \in S$.
\end{enumerate}
\end{PROP}

To establish Theorem~\ref{theo:limit-obj}, we begin with a lemma that
all the iterate points remain in a subset of $\FC$.

\begin{LEMM}\label{lemm:level}
Let $\LC$ be the level set of $g$ determined by the initial value 
$g(\y^0, \W^0)$,
\begin{eqnarray*}
  \LC := \{(\y, \W) \in \Real^m \times \SMAT^n :
  (\y, \W) \in \FC, \ g(\y, \W) \ge g(\y^0, \W^0) \}.
\end{eqnarray*} 
Then, the sequence $\{(\y^k, \W^k)\}$ generated by Algorithm~\ref{algo:dspg} satisfies $(\y^k, \W^k) \in \LC$ for each $k$.
\end{LEMM}

\begin{proof}
First, we prove that $(\y^k, \W^k) \in \FC$ for each $k$.
By the assumption (iii), 
we have $(\y^0, \W^0) \in \FC$.
Assume that $(\y^k, \W^k) \in \FC$ for some $k$.
Since $0 \le \lambda^k \le 1$ in Step 4 and $\W^k \in \WC$,
the convexity of $\WC$ indicates
$\W^{k+1} = \W^k + \lambda^k \Delta \W^k
= (1 - \lambda^k) \W^k
+ \lambda^k [\W^k + \alpha^k \X^k]_{\le \bold{\rho}} \in \WC.$
In addition, the value $\theta$ of Step 3 ensures
$\C + (\W^k + \lambda \Delta \W^k) - \AC^T(\y^k + \lambda \Delta \y^k) \succ \O$ 
for $\lambda \in [0, \overline{\lambda}^k]$.
Hence, $(\y^{k+1}, \W^{k+1}) \in \FC$.

Now, we verify that
$g(\y^k, \W^k) \ge g(\y^0, \W^0)$
for each $k$.
The case $k=0$ is clear.
The case $k \ge 1$ depends on the fact 
$(\Delta \y^k, \Delta \W^k)$ is an ascent direction
of $g$ at $(\y^k, \W^k)$;
\begin{eqnarray}
  & & \nabla g (\y^k, \W^k) \bullet (\Delta \y^k, \Delta \W^k) 
\nonumber \\
  &=& (\nabla_{\y} g (\y^k, \W^k),
  \nabla_{\W} g (\y^k, \W^k))
  \bullet (\Delta \y^k, \Delta \W^k) \nonumber \\
  &=& \alpha^k || \b - \AC(\X^k) ||^2
    + \X^k \bullet ([\W^k + \alpha^k \X^k]_{\le \bold{\rho}} -
  \W^k) \nonumber \\
  &\ge&  \alpha^k || \b - \AC(\X^k) ||^2
  + \frac{1}{\alpha^k} ||[\W^k + \alpha^k \X^k]_{\le \bold{\rho}} - \W^k||^2 
  \nonumber \\
  &=& \frac{1}{\alpha^k} || (\Delta \y^k, \Delta \W^k) ||^2
  \label{eq:inequality-g}\\
  &\ge& 0 \nonumber .
\end{eqnarray}
The first inequality comes from (P2)
by putting $\WC$ as $S$, $\W^k + \alpha^k \X^k$ as $\x$ and $\W^k$ as $\y$,
and using the relations 
$\P_{\WC}(\W^k+\alpha^k \X^k) = [\W^k+\alpha^k \X^k]_{\le \bold{\rho}}$ 
and
$\P_{\WC}(\W^k) = \W^k$ by $(\y^k, \W^k) \in \FC$.

When the inner iteration terminates, we have 
\begin{eqnarray*}
g(\y^{k+1}, \W^{k+1})
&\ge& \min_{0 \le h \le \min\{k, M-1\}}
g(\y^{k-h}, \W^{k-h}) + \gamma \lambda^k \nabla g(\y^k, \W^k)
\bullet (\Delta \y^k, \Delta \W^k) \\
  &\ge& 
  \min_{0 \le h \le \min\{k, M-1\}}
  g(\y^{k-h}, \W^{k-h}).
\end{eqnarray*}
Therefore, if
$\min_{0 \le h \le k} g(\y^h, \W^h) \ge g(\y^0, \W^0)$,
we obtain
$g(\y^{k+1}, \W^{k+1}) \ge g(\y^0, \W^0)$.
By induction, we conclude $(\y^k, \W^k) \in \LC$ for each $k$.
\end{proof}

The key to establishing Theorem~\ref{theo:limit-obj}
is the boundedness of the level set $\LC$.
\begin{LEMM}\label{lemm:bound}
The level set $\LC$ is bounded.
\end{LEMM}

\begin{proof}
If $(\y, \W) \in \LC$, then $\W \in \WC$. Thus, 
the boundedness of $\W$ is clear from $|W_{ij}| \le \rho_{ij}$.
We then fix $\widehat{\W} \in \WC$ and show the boundedness of
\begin{eqnarray*}
  \LC_{\widehat{\W}} := \{\y \in \Real^m :
g(\y, \widehat{\W}) \ge g(\y^0, \W^0),
\quad
\C + \widehat{\W} - \AC^T(\y) \succ \O\}.
\end{eqnarray*}
Let $\Z := \C + \widehat{\W} - \AC^T(\y)$ for $\y \in \LC_{\widehat{\W}}$.
Since $\AC$ is surjective,
the map $\AC \AC^T : \Real^m \to \Real^m$ is nonsingular and
\begin{eqnarray*}
  ||\y|| = ||(\AC \AC^T)^{-1} \AC(\C + \widehat{\W} - \Z)||
  \le ||(\AC \AC^T)^{-1}|| \cdot ||\AC|| \cdot (||\C|| + ||\widehat{\W}|| + ||\Z||).
\end{eqnarray*}
Hence, if we can prove the boundedness of $\Z$, the desired result follows.

Since we assume that $(\PC)$ has at least one interior point,
there exists $\widehat{\X}$ such that $\AC(\widehat{\X}) = \b$ and $\widehat{\X} \succ \O$.
We denote the eigenvalues of $\Z$ by
$0 < \lambda_1(\Z) \le \lambda_2(\Z) \le \cdots \le \lambda_n(\Z)$.
For simplicity, we use $\lambda_{\min} (\Z) := \lambda_1(\Z)$ and
$\lambda_{\max} (\Z) := \lambda_n(\Z)$.
Letting $\bar{c}_0 :=  g(\y^0, \W^0) - n \mu + n \mu \log \mu$, we  can derive equivalent inequalities from $g(\y, \widehat{\W}) \ge g(\y^0, \W^0)$;
\begin{eqnarray*}
& & g(\y, \widehat{\W}) \ge g(\y^0, \W^0) \\
& \Leftrightarrow & \b^T \y + \mu \log \det (\C + \widehat{\W} - \AC^T(\y)) \ge 
\bar{c}_0 \\
&\Leftrightarrow & \AC(\widehat{\X})^T\y + \mu \log \det \Z \ge \bar{c}_0 \\
&\Leftrightarrow & \widehat{\X}\bullet\AC^T(\y) + \mu \log \det \Z \ge \bar{c}_0 \\
&\Leftrightarrow & \widehat{\X} \bullet (\C + \widehat{\W} - \Z) + \mu \log \det \Z \ge \bar{c}_0 \\
&\Leftrightarrow & \widehat{\X} \bullet \Z - \mu \log \det \Z 
\le -\bar{c}_0 + \widehat{\X} \bullet (\C + \widehat{\W})
\end{eqnarray*}
Since $\widehat{\X} \bullet \widehat{\W} = 
\sum_{i=1}^n \sum_{j=1}^n \widehat{X}_{ij}  \widehat{W}_{ij}
\le
\sum_{i=1}^n \sum_{j=1}^n |\widehat{X}_{ij}|  \rho_{ij} 
= |\widehat{\X}| \bullet \bold{\rho}$, it holds that 
$\widehat{\X} \bullet \Z - \mu \log \det \Z 
\le c$,
where $c :=  - \bar{c}_0 + \widehat{\X} \bullet \C
+ |\widehat{\X}| \bullet \bold{\rho}$.
From  $\min_t \{at -\log t : t>0 \} = 1 + \log a$ for
any $a > 0$, it follows that
\begin{eqnarray*}
  \widehat{\X} \bullet \Z - \mu \log \det \Z
  &\ge& \sum_{i=1}^n [\lambda_{\min}(\widehat{\X}) \lambda_i(\Z)
  - \mu \log \lambda_i (\Z)] \\
  &\ge& (n-1) \mu \left(1 + \log \frac{\lambda_{\min}(\widehat{\X})}{\mu}\right)
  + \lambda_{\min}(\widehat{\X}) \lambda_{\max}(\Z) - \mu \log \lambda_{\max} (\Z).
\end{eqnarray*}
Hence, 
\begin{eqnarray*}
  \lambda_{\min}(\widehat{\X}) \lambda_{\max}(\Z)
- \mu \log \lambda_{\max}(\Z) 
\le 
c - (n-1) \mu \left(1 + \log \frac{\lambda_{\min}(\widehat{\X})}{\mu}\right).
\end{eqnarray*}
Note that the right-hand side is determined by only $\widehat{\X}$ and is independent from $\Z$, and that $\lambda_{\min}(\widehat{\X}) > 0$ from $\widehat{\X} \succ \O$.
Hence, there exists $\beta_{\sZ}^{\max} < \infty$ such that
$\lambda_{\max}(\Z) \le \beta_{\sZ}^{\max}$ for all
$(\y, \widehat{\W}) \in \LC$.

In addition,
from $\widehat{\X} \bullet \Z - \mu \log \det \Z \le c$ and $\widehat{\X} \bullet \Z \ge 0$,
we have
\begin{eqnarray*}
  \log \det \Z &\ge& -\frac{c}{\mu} \\
  \log \lambda_{\min}(\Z) &\ge&
  -\frac{c}{\mu}  - \sum_{i=2}^{n} \log \lambda_i(\Z)
    \ge  -\frac{c}{\mu}  - (n-1) \log \beta_{\sZ}^{\max} \\
  \lambda_{\min}(\Z) &\ge&
  \beta_{\sZ}^{\min}
  := \exp\left(-\frac{c}{\mu}  - (n-1) \log \beta_{\sZ}^{\max}
      \right) > 0.
\end{eqnarray*}
Therefore, the minimum and maximum eigenvalues of $\Z$ are bounded
for $(\y, \widehat{\W}) \in \LC$.
This completes the proof.
\end{proof}

\begin{REMA}\label{rem:yw-bound} \rm
From Lemmas~\ref{lemm:level} and \ref{lemm:bound},
$||\y^k||$ and $||\W^k||$ are bounded;
$||\y^k|| \le \eta_{\y} := ||(\AC\AC^T)^{-1}||\cdot||\AC||\cdot
(||\C||+||\bold{\rho}||+\sqrt{n}\beta_{\sZ}^{\max})$ and 
$||\W^k|| \le \eta_{\sW} := ||\bold{\rho}||$.
\end{REMA}

\begin{REMA}\label{rema:x-bound} \rm
Lemma~\ref{lemm:bound} implies that the set 
$\{\X(\y, \W) : (\y, \W) \in \LC\}$ is also bounded.
If we denote
$\beta_{\sX}^{\min} := \frac{\mu}{\beta_{\sZ}^{\max}} > 0$
and
$\beta_{\sX}^{\max} := \frac{\mu}{\beta_{\sZ}^{\min}} < \infty$,
then we have
$\beta_{\sX}^{\min} \I \preceq \X(\y, \W) \preceq \beta_{\sX}^{\max} \I$
for $(\y, \W) \in \LC$.
In particular, since $(\y^k, \W^k) \in \LC$ from Lemma~\ref{lemm:level},
$\X^k = \X(\y^k,\W^k) = \mu (\C + \W - \AC^T(\y^k))^{-1}$ is also bounded;
$\beta_{\sX}^{\min} \I \preceq \X^k \preceq \beta_{\sX}^{\max} \I$.
Furthermore,  for $(\y, \W) \in \LC$, we obtain the bounds
$||\X(\y, \W)|| \le \eta_{\sX}$ and
$||\X^{-1}(\y, \W)|| \le \eta_{\sX^{-1}}$,
where
$\eta_{\sX} := \sqrt{n} \beta_{\sX}^{\max}  > 0$ 
and $\eta_{\sX^{-1}} := \frac{\sqrt{n}}{\beta_{\sX}^{\min}}  > 0$.
Hence, it holds that 
$||\X^k|| \le \eta_{\sX}$ and $||(\X^k)^{-1}|| \le \eta_{\sX^{-1}}$
for each $k$.
\end{REMA}

\begin{REMA}\label{rema:yw-bound} \rm
It follows from Remark~\ref{rema:x-bound} that
$||\Delta \y^k||$ and $||\Delta \W^k||$ are also bounded by
$\eta_{\Delta \y} : = \alpha_{\max}(||\b|| + ||\AC|| \eta_{\X})$
and $\eta_{\Delta \sW} := \alpha_{\max} \eta_{\sX}$, respectively.
These bounds are found by
\begin{eqnarray*}
  ||\Delta \y^k|| &=& ||\alpha^k (\b - \AC(\X^k)) ||
  \le \alpha^k (||\b|| + ||\AC||\cdot ||\X^k||)
  \le \alpha_{\max} (||\b|| + ||\AC|| \eta_{\sX}) \\
  ||\Delta \W^k|| &=&
  ||[\W^k + \alpha^k \X^k]_{\le \bold{\rho}} - \W^k||
  \le || \alpha^k \X^k || \le \alpha_{\max} \eta_{\sX}.
  \end{eqnarray*}
For $||\Delta \W^k||$, we substitute $S = \WC$,
$\x = \W^k + \alpha^k \X^k$ and $\y = \W^k = \P_{\sWC} (\W^k)$ to (P3).
\end{REMA}

From Lemma~\ref{lemm:bound}, the set of the optimal solutions $\FC^*$ is a subset of
$\{(\y, \W) \in \Real^m \times \SMAT^n : |\W| \le \bold{\rho}, \ 
\beta_{\sZ}^{\min} \I \preceq \C + \W - \AC^T(\y)  
\preceq \beta_{\sZ}^{\max}\I \}$
and it is a closed convex set and bounded.  
From the continuity of the objective
function $g$, the dual problem $(\DC)$ has an optimal solution.
Furthermore, since both $(\PC)$ and $(\DC)$ has an interior feasible point,
the duality theorem holds~\cite{BORWEIN06, BOYD04}, that is,
the primal problem $(\PC)$ 
also has an optimal solution and there is no duality gap between
$(\PC)$ and $(\DC)$,  $f^* = g^*$.
In the following Lemma~\ref{lemm:optset}, we show the uniqueness of 
the optimal solution in $(\PC)$ and a property of 
the optimal solutions in  $(\DC)$.

\begin{LEMM}\label{lemm:optset}
The optimal solution of $(\PC)$ is unique.
In addition, if both $(\y_1^*, \W_1^*)$ and $(\y_2^*, \W_2^*)$
are optimal solutions of $(\DC)$, then
$\X(\y_1^*,\W_1^*) = \X(\y_2^*,\W_2^*)$ and
$\b^T \y_1^* = \b^T \y_2^*$.
\end{LEMM}

\begin{proof}
Since the function $-\log \det \X$ is strictly convex
\cite{BOYD04}, 
we have
\begin{eqnarray}
-\log \det \left(\frac{\X_1 + \X_2}{2}\right)
< -\frac{1}{2} \log \det \X_1
-\frac{1}{2} \log \det \X_2
\quad \mbox{for} \quad \forall \X_1 \succ \O , \forall \X_2 \succ \O
(\X_1 \ne \X_2).
\label{eq:strcitly-convex}
\end{eqnarray}

Suppose that we have two different optimal solutions $(\y_1^*, \W_1^*)$ and
$(\y_2^*, \W_2^*)$ for $(\DC)$ such that
$\C + \W_1^* - \AC^T(\y_1^*)  \ne \C + \W_2^* - \AC^T(\y_2^*)$.
Since $(\y_1^*, \W_1^*)$ and $(\y_2^*, \W_2^*)$ attain the same objective value, it holds that
$ g^* = \b^T \y_1^*  + \mu \log \det (\C + \W_1^* - \AC^T(\y_1^*)) + n \mu - n \mu \log \mu
= \b^T \y_2^*  + \mu \log \det (\C + \W_2^* - \AC^T(\y_2^*))  + n \mu - n \mu \log \mu$.
Since the feasible set of $(\DC)$ is convex,
$\left(\frac{\y_1^* + \y_2^*}{2}, \frac{\W_1^* + \W_2^*}{2}\right)$
is also feasible.
However, the inequality (\ref{eq:strcitly-convex}) indicates
\begin{eqnarray*}
& & \b^T \left(\frac{\y_1^* + \y_2^*}{2}\right)
+  \mu \log \det \left(\C + \frac{\W_1^*+\W_2^*}{2}
          + \AC^T\left(\frac{\y_1^*+\y_2^*}{2}\right)\right) + n \mu - n \mu \log \mu \\
&>& \frac{1}{2}
\left(\b^T \y_1^* + \mu \log \det (\C + \W_1^* - \AC^T(\y_1^*)) + n \mu - n \mu \log \mu \right) \\
& &  + \frac{1}{2}
\left(\b^T \y_2^* + \mu \log \det (\C + \W_2^* - \AC^T(\y_2^*)) + n \mu - n \mu \log \mu \right)
= \frac{g^*}{2} + \frac{g^*}{2}
= g^*.
\end{eqnarray*}
This is a contradiction to the optimality of $g^*$.
Hence, we obtain
$\C + \W_1^* - \AC^T(\y_1^*)  = \C + \W_2^* -\AC^T(\y_2^*)$,
which is equivalent to
$\X(\y_1^*, \W_1^*) = \X(\y_2^*, \W_2^*)$.
Since the objective values of both  $(\y_1^*, \W_1^*)$ and
$(\y_2^*, \W_2^*)$ are $g^*$,
it is easy to show  $\b^T \y_1^* = \b^T \y_2^*$ from  $\C + \W_1^* - \AC^T(\y_1^*)  = \C + \W_2^* -\AC^T(\y_2^*)$.

The uniqueness of optimal solution in $(\PC)$
can also be obtained by the same argument using (\ref{eq:strcitly-convex}).
\end{proof}

Next, we examine the validity of the stopping criteria in
Algorithm~\ref{algo:dspg}.
\begin{LEMM}\label{lemm:stop}
$(\y^*, \W^*)$ is optimal for $(\DC)$
if and only if $(\y^*, \W^*) \in \FC$ and
\begin{eqnarray}
\P_{\widehat{\sWC}}((\y^*, \W^*)
+ \alpha \nabla g(\y^*, W^*)) = (\y^*, \W^*)
\label{eq:opt-condition}
\end{eqnarray}
for some $\alpha > 0$.
\end{LEMM}

As proven in \cite{HAGER08}, for a general convex problem
\begin{eqnarray*}
\min f_1(\x) \quad \mbox{s.t.} \quad \x \in S_1
\end{eqnarray*}
with a differentiable convex function $f_1:\Real^n \to \Real$
and a closed convex set $S_1 \subset \Real^n$,
a point $\x^* \in S_1$ is optimal if and only if
$\P_{S_1}(\x^* - \alpha \nabla f_1(\x^*)) = \x^*$ for some $\alpha > 0$.
This condition is further extended to 
$\P_{S_1}(\x^* - \alpha \nabla f_1(\x^*)) = \x^*$ for {\it any} $\alpha > 0$.
This results cannot be applied to $(\DC)$
since the projection onto the intersection
$\FC = \widehat{\WC} \cap \widehat{\FC}$ is not available
at a low computation cost.  The projection considered in the proposed method is onto
$\widehat{\WC}$, thus we prove   
Lemma~\ref{lemm:stop} as follows.

\begin{proof}
It is easy to show that
the condition (\ref{eq:opt-condition})
for some $\alpha > 0$ is equivalent to (\ref{eq:opt-condition})
for any $\alpha > 0$,
following the proof for the condition (P6) of \cite{HAGER08}.

We now suppose that $(\y^*, \W^*) \in \FC$ and 
$\P_{\widehat{\WC}}((\y^*, \W^*)
+ \alpha \nabla g(\y^*, \W^*)) = (\y^*, \W^*)$
for any $\alpha > 0$.
Let $\X^* :=  \X(\y^*, \W^*) = \mu (\C + \W^* - \AC^T(\y^*))^{-1}$.
By considering the definitions of $\P_{\widehat{\WC}}$ and $\nabla g$
into (\ref{eq:opt-condition}), we have
two equalities $\AC(\X^*) = \b$ and
$[\W^* + \alpha \X^*]_{\le \bold{\rho}} = \W^*$.
Since $\X^* \succ O$, $\X^*$ is a feasible point of $(\PC)$.
The second equality $[\W^* + \alpha \X^*]_{\le \bold{\rho}} = \W^*$
indicates the three cases:
\begin{enumerate}
  \item[Case 1] ($X^*_{ij} > 0$) : 
      There exists $\alpha > 0$ such that 
      $W^*_{ij} + \alpha X^*_{ij} > \rho_{ij}$.
      From $[\W^* + \alpha \X^*]_{\le \bold{\rho}} = \W^*$,
      we obtain  $W^*_{ij} = \rho_{ij}$ .
  \item[Case 2] ($X^*_{ij} < 0$) :
      In a similar way to Case 1, we obtain $W^*_{ij} = -\rho_{ij}$.
  \item[Case 3] ($X^*_{ij} = 0$) :
      In this case, we know only $|W^*_{ij}| \le \rho_{ij}$.
\end{enumerate}

Using the relations $\X^* = \mu (\C + \W^* - \AC^T(\y^*))^{-1}$ and
$\AC(\X^*) = \b$,
we consider the difference of the primal and dual objective
functions,
\begin{eqnarray}
& & f(\X^*) - g(\y^*,\W^*) \nonumber \\
&=& \left(\C \bullet \X^* - \mu \log \det \X^*
  +  \bold{\rho} \bullet |\X^*| \right) \nonumber \\
& & - \left(\b^T \y^* + \mu \log \det (\C + \W^* - \AC^T(\y^*))
    + n \mu (1-\log \mu) \right) \nonumber \\
&=& 
  \bold{\rho} \bullet |\X^*| - \W^* \bullet \X^* \label{eq:primal-dual}
\end{eqnarray}
The above three cases imply that this difference is 0.
Note that $\X^*$ and $(\y^*, \W^*)$ are feasible for
$(\PC)$ and $(\DC)$, respectively, and there is no duality gap, hence,
$\X^*$ and $(\y^*, \W^*)$ are optimal for $(\PC)$ and $(\DC)$.

For the converse, we suppose that $(\y^*, \W^*)$ is 
an optimal solution of $(\DC)$.
Again, let $\X^* = \mu (\C + \W^* - \AC^*(\y^T))^{-1}$.
Since $(\DC)$ is a concave maximization problem,
$(\y^*,\W^*)$ satisfies
\begin{eqnarray*}
\nabla g(\y^*,\W^*) \bullet ((\y,\W) - (\y^*,\W^*)) \le 0
\quad \mbox{for} \quad \forall (\y, \W) \in \FC,
\end{eqnarray*}
or equivalently,
\begin{eqnarray}
(\b - \AC(\X^*))^T (\y - \y^*)
  + \X^* \bullet (\W - \W^*) \le 0
  \quad \mbox{for} \quad \forall (\y, \W) \in \FC. \label{eq:opt-conditions2}
\end{eqnarray}
Since $\C + \W^* - \AC^T(\y^*) \succ \O$ and $\AC^T$ is a continuous map,
there is a small  $t > 0$ such that 
$\C + \W^* - \AC^T(\y^* + t (\b - \AC(\X^*))) \succ \O$.
Therefore $(\y^* + t (\b - \AC(\X^*)), \W^*)$ is feasible,
and when we put $(\y^* + t (\b - \AC(\X^*)), \W^*) \in \FC$ into
$(\y, \W)$ of (\ref{eq:opt-conditions2}), we obtain $\AC(\X^*) = \b$.
Hence, we have
$\y^* + \alpha (\b - \AC(\X^*) ) = \y^*$.
Similarily, when we perturb $\W^*$ in element-wise, 
we obtain two indications;
if $\X^*_{ij} > 0$ then $\W^*_{ij} = \rho_{ij}$
and 
if $\X^*_{ij} < 0$ then $\W^*_{ij} = -\rho_{ij}$.
This leads to the results
$[\W^* + \alpha \X^*]_{\le \bold{\rho}} = \W^*$.
Hence, we have shown that  (\ref{eq:opt-condition}) holds
for $\forall \alpha > 0$.
\end{proof}

From Lemma~\ref{lemm:stop} and Lemma~\ref{lemm:optset},
we also find the relation of the optimal solutions of $(\PC)$ and $(\DC)$.

\begin{REMA}\label{rema:x-optimal} \rm
The matrix $\X^*$ computed by $\X^* := \X(\y^*, \W^*)$
for an optimal solution $(\y^*, \W^*)$ of $(\DC)$
is the unique optimal solution of $(\PC)$.
Furthermore, from $(\y^*, \W^*) \in \LC$ and
Remark~\ref{rema:x-bound},
the optimal solution $\X^*$ satisfies
$\beta_{\sX}^{\min} \I \preceq \X^* \preceq \beta_{\sX}^{\max} \I$ and
$||\X^*|| \le \eta_{\sX}$.
\end{REMA}

From the definition in (\ref{eq:search-direction}),
$(\Delta \y^k, \Delta \W^k)$ depends on $\alpha^k$.
However,  the stopping criteria shown in Lemma~\ref{lemm:stop}
is practically independent of $\alpha^k$.
For the subsequent analysis, we introduce
$(\Delta \y^k_{(1)}, \Delta \W^k_{(1)})$ by setting $\alpha^k = 1$;
\begin{eqnarray}
(\Delta \y^k_{(1)}, \Delta \W^k_{(1)}) &:=&
  \P_{\widehat{\WC}}
  ((\y^k, \W^k) + \nabla g(\y^k, \W^k))
  - (\y^k, \W^k) \nonumber \\
&=& (\b - \AC(\X^k),
  [\W^k + \X^k]_{\le \bold{\rho}} - \W^k).
  \label{eq:search-direction1}
\end{eqnarray}
and we now investigate the relation between $(\Delta \y^k, \Delta \W^k)$
and $(\Delta \y^k_{(1)}, \Delta \W^k_{(1)})$.

\begin{LEMM}\label{lemm:direction1}
The search direction $(\Delta \y^k, \Delta \W^k)$ is
bounded by 
$(\Delta \y^k_{(1)}, \Delta \W^k_{(1)})$. More precisely, 
\begin{eqnarray}
\min\{1, \alpha_{\min}\} || (\Delta \y^k_{(1)}, \Delta \W^k_{(1)})||
  \le 
  || (\Delta \y^k, \Delta \W^k) ||
  \le 
\max\{1, \alpha_{\max}\} || (\Delta \y^k_{(1)}, \Delta \W^k_{(1)})||.
\label{eq:norm-bound}
  \end{eqnarray}
\end{LEMM}

\begin{proof}
It holds that $\Delta \y^k = \alpha^k \Delta \y_{(1)}^k$ from the definitions.
From (P4) of Proposition~\ref{prop:hager}, we know that  
$ ||\P_{\WC} (\W^k + \alpha \X^k) - \W^k||$ is non-decreasing
for $\alpha > 0$, therefore, it holds for the case $\alpha^k > 1$ that 
$||\Delta \W^k|| = ||[\W^k + \alpha^k \X^k]_{\le \bold{\rho}} - \W^k||
\ge ||[\W^k + \X^k]_{\le \bold{\rho}} - \W^k|| = ||\Delta \W^k_{(1)}||$.
In addition, 
(P5) of Proposition~\ref{prop:hager} indicates that  
$ ||\P_{\WC} (\W^k + \alpha \X^k) - \W^k|| / \alpha $ is non-increasing
for $\alpha > 0$.
Since we choose $\alpha^k$ from $[\alpha_{\min}, \alpha_{\max}]$,
we have
$||\Delta \W^k|| = || [\W^k + \alpha^k \X^k]_{\le \bold{\rho}} - \W^k|| 
\ge  \alpha^k ||[\W^k + \X^k]_{\le \bold{\rho}} - \W^k|| 
\ge \alpha_{\min} ||[\W^k + \X^k]_{\le \bold{\rho}} - \W^k||
= \alpha_{\min} ||\Delta \W^k_{(1)}||$
for the case $\alpha^k \le 1$.
The combination of these two shows  the left inequality
of (\ref{eq:norm-bound}).
The right inequality is also derived from (P4) and (P5) in a similar way.

\end{proof}

The condition $||(\Delta \y^k, \Delta \W^k)|| > 0$ can be assumed without loss of generality,
since  $||(\Delta \y^k, \Delta \W^k)|| = 0$ indicates that $(\y^k, \W^k)$ 
is an optimal solution by Lemmas~\ref{lemm:stop} and \ref{lemm:direction1} 
and (\ref{eq:search-direction1}) and that Algorithm~\ref{algo:dspg} stops at 
Step~2.

Algorithm~\ref{algo:dspg} may terminate before computing an approximate solution with
a required accuracy in the following two cases:
(i) The step length $\lambda^k$ converges to 0 before 
$||(\Delta \y^k, \Delta \W^k)||$ reaches 0, and $(\y^k, \W^k)$ cannot proceed,
(ii) The norm of the search direction $||(\Delta \y^k, \Delta \W^k)||$ converges to 0 
before $g(\y^k, \W^k)$ reaches the optimal value $g^*$.
Lemmas~\ref{lemm:low-min} and \ref{lemm:liminf-obj} 
show that the two cases will not happen.  
For the proofs of the two lemmas,
we first discuss some inequalities related to matrix norms.

\begin{LEMM}\label{lemm:matrix-norm}
Suppose that  $ 0 < \widehat{\beta}^{\min} < \widehat{\beta}^{\max} < \infty$. 
For $\forall \X, \forall \Y \in
S_2 := \{ \X \in \SMAT^n :
\widehat{\beta}^{\min} \I \preceq \X \preceq \widehat{\beta}^{\max} \I\}$,
it holds 
\begin{enumerate}
  \item[(i)] $(\Y-\X) \bullet (\X^{-1} - \Y^{-1}) \ge
        \frac{1}{(\widehat{\beta}^{\max})^2} || \Y-\X ||^2$,
  \item[(ii)] $(\Y-\X) \bullet (\X^{-1} - \Y^{-1}) \ge
        (\widehat{\beta}^{\min})^2 || \Y^{-1}-\X^{-1} ||^2$,
  \item[(iii)] $|| \Y-\X || \ge
          (\widehat{\beta}^{\min})^2|| \Y^{-1}-\X^{-1} ||$.
\end{enumerate}
\end{LEMM}

\begin{proof}
From the discussions of \cite{dASPREMONT08}, the function $f_2 (\X) = -\log \det (\X) $ is
strongly convex  with the convexity parameter
$\frac{1}{2(\widehat{\beta}^{\max})^2}$ on the set $S_2$.
Therefore, it holds that 
\begin{eqnarray}
  f_2(\Y) \ge f_2(\X) + \nabla f_2(\X) \bullet (\Y - \X)
  + \frac{1}{2(\widehat{\beta}^{\max})^2} || \Y - \X||^2 \label{eq:logdet}
  \end{eqnarray}
for $\forall \X, \forall \Y \in S_2$.
By swapping $\X$ and $\Y$, we also have
\begin{eqnarray*}
  f_2(\X) \ge f_2(\Y) + \nabla f_2(\Y) \bullet (\X - \Y)
  + \frac{1}{2(\widehat{\beta}^{\max})^2} || \X - \Y||^2.
  \end{eqnarray*}
Since $\nabla f_2(\X) = -\X^{-1}$, adding these two inequalities
generates (i). When we use
$\X^{-1}, \Y^{-1} \in \{\X :
\frac{1}{\widehat{\beta}^{\max}} \I \preceq \X
\preceq \frac{1}{\widehat{\beta}^{\min}} \I\}$, 
we obtain (ii) in a similar way to (i).
Finally,  an application of the Cauchy-Schwartz inequality to (ii) lead to 
\begin{eqnarray*}
  (\widehat{\beta}^{\min})^2 || \Y^{-1}-\X^{-1} ||^2 \le (\Y-\X) \bullet (\X^{-1} - \Y^{-1})
  \le ||\Y-\X || \cdot ||\X^{-1} - \Y^{-1}||.
\end{eqnarray*}
If $\X \ne \Y$, (iii) is obtained by dividing the both sides with
$||\X^{-1} - \Y^{-1}||$, meanwhile if $\X = \Y$, (iii) is obvious.
\end{proof}

\begin{LEMM}\label{lemm:low-min}
The step length $\lambda^k$ of Algorithm~\ref{algo:dspg}
has a lower bound,
\begin{eqnarray*}
  \lambda^k \ge
  \min\left\{ \overline{\lambda}_{\min},
    \frac{2\sigma_1(1-\gamma)}{L{\scriptsize \alpha}_{\max}} \right\}
\end{eqnarray*}
where
$\overline{\lambda}_{\min} :=
  \min\left\{1,
\frac{\beta_{\sZ}^{\min}\tau}{\eta_{\Delta \sW} + ||\AC^T|| \eta_{\Delta \y}}\right\}$
and $L := \frac{ \mu \sqrt{2(||\AC||^2+1)}\max\{1,||\AC||\}}
{((1-\tau)\beta_{\sZ}^{\min})^2}$.
\end{LEMM}

\begin{proof}
We first show the lower bound of $\overline{\lambda}^k$ of Step 3.
Since $\overline{\lambda}^k$ is determined by (\ref{eq:overlabmda}),
we examine a bound of $\lambda$ such that
$\Z(\lambda) := \C + (\W^k + \lambda \Delta \W^k) - \AC^T(\y^k + \lambda \Delta \y^k)
\succeq \O$.
It follows from Remark~\ref{rema:x-bound}
that $\mu(\X^k)^{-1} \succeq \beta_{\sZ}^{\min} \I$.
From Remark~\ref{rema:yw-bound}, we also have
$||\Delta \y^k|| \le \eta_{\Delta \y}$
and $||\Delta \W^k|| \le \eta_{\Delta \W}$.
Therefore, we obtain
\begin{eqnarray}
  \Z(\lambda)
&=& \mu (\X^k)^{-1} + \lambda (\Delta \W^k - \AC^T(\Delta \y^k))
\nonumber \\
  &\succeq&  \beta_{\sZ}^{\min} \I
  - \lambda (\eta_{\Delta \W} + ||\AC^T|| \eta_{\Delta \y}) \I.
  \label{eq:Z_labmda}
\end{eqnarray}
Hence, for any
$\lambda \in \left[ 0, \frac{\beta_{\sZ}^{\min}}{\eta_{\Delta \W} + ||\AC^T||
\eta_{\Delta \y}} \right]$,
we have
$\Z(\lambda) \succeq \O$,
and consequently, we obtain
$\overline{\lambda}^k \ge \overline{\lambda}_{\min}$.

If $\theta$ of (\ref{eq:overlabmda}) is non-negative,
$\Z(\lambda) \succeq \Z(0) \succeq (1-\tau) \Z(0)$.
In the case  $\theta < 0$, we have $\overline{\lambda}^k \ge -\frac{1}{\theta} \times \tau$, and this leads to 
$\Z(\lambda) \succeq (1-\tau) \Z(0)$ for $\lambda\in[0,\overline{\lambda}^k]$. Therefore, 
$\Z(\lambda) \succeq (1-\tau) \Z(0) \succeq (1-\tau)\beta_{\sZ}^{\min} \I$
for $\lambda \in [0, \overline{\lambda}^k]$.
Hence,  
it follows from (iii) of Lemma~\ref{lemm:matrix-norm} that 
\begin{eqnarray*}
  ||\Z(\lambda)^{-1} - \Z(0)^{-1}|| \le
  \frac{||\Z(\lambda) - \Z(0)||}{((1-\tau)\beta_{\sZ}^{\min})^2}
  \quad \mbox{for} \quad \lambda \in [0, \overline{\lambda}^k].
\end{eqnarray*}
Hence, we acquire some Lipschitz continuity on $\nabla g$
for the direction $(\Delta \y^k, \Delta \W^k)$.
For $\lambda \in [0, \overline{\lambda}^k]$, we have
\begin{eqnarray}
  & & || \nabla g(\y^k + \lambda \Delta \y^k,
  \W^k + \lambda \Delta \W^k)
  - \nabla g(\y^k, \W^k) ||  \nonumber \\
  & = &
  \left|\left|
      \left( \b - \AC(\mu\Z(\lambda)^{-1}), \mu \Z(\lambda)^{-1}) \right) 
      - \left( \b - \AC(\mu\Z(0)^{-1}), \mu \Z(0)^{-1}) \right)
      \right| \right|
  \nonumber \\
  & = & 
  \mu \left|\left|
      \left(-\AC(\Z(\lambda)^{-1})+\AC(\Z(0)^{-1}),
        \Z(\lambda)^{-1} - \Z(0)^{-1} \right) \right| \right|
  \nonumber \\
  & \le &
  \mu \sqrt{||\AC||^2 + 1} || \Z(\lambda)^{-1} -  \Z(0)^{-1}||
  \nonumber \\
  & \le &
  \frac{ \mu \sqrt{||\AC||^2 + 1} }{((1-\tau)\beta_{\sZ}^{\min})^2}
  ||\Z(\lambda) - \Z(0) || \nonumber \\
  & = &
  \frac{\lambda \mu \sqrt{||\AC||^2 + 1} }{((1-\tau)\beta_{\sZ}^{\min})^2}
  ||\Delta \W^k - \AC^T(\Delta \y^k)|| \nonumber \\
  & \le &
  \frac{\lambda \mu \sqrt{2(||\AC||^2 + 1)} \max\{1,||\AC||\}}
  {((1-\tau)\beta_{\sZ}^{\min})^2}
  ||(\Delta \y^k, \Delta \W^k)|| \nonumber \\
  & = & 
  \lambda L ||(\Delta \y^k, \Delta \W^k)||,
\label{eq:g-contium}
\end{eqnarray}  
Here, we have used  the inequalities
$||\Delta \W^k - \AC^T(\Delta \y^k)|| \le
||\Delta \W^k|| + ||\AC^T || \cdot ||\Delta \y^k||$ 
and $||\Delta \W^k||+ ||\Delta \y^k|| \le
\sqrt{2} || (\Delta \y^k, \Delta \W^k)||$.

We examine how the inner loop, Step 3 of Algorithm 2.1, is executed.
As in the Armijo rule, the inner loop
terminates at a finite number of inner iterations.
If (\ref{eq:g-condition}) is satisfied at $j=1$, then
$\lambda^k = \overline{\lambda}^k \ge \overline{\lambda}_{\min}$.
If (\ref{eq:g-condition}) is satisfied at $j \geq 2$, 
then (\ref{eq:g-condition}) is not satisfied  at $j-1$. Thus, we have
\begin{eqnarray*}
  & &
  g(\y^k + \lambda_{j-1}^k \Delta \y^k,
  \W^k + \lambda_{j-1}^k \Delta \W^k) \\
  &<& \min_{0 \le h \le \min\{k, M-1\}} g(\y^{k-h}, \W^{k-h}) +
  \gamma 
  \lambda_{j-1}^k
  \nabla g (\y^k, \W^k) \bullet (\Delta \y^k, \Delta \W^k)
  \\
  &\le& 
  g(\y^{k}, \W^{k}) +
  \gamma 
  \lambda_{j-1}^k
  \nabla g (\y^k, \W^k) \bullet (\Delta \y^k, \Delta \W^k).
\end{eqnarray*}

From Taylor's expansion and (\ref{eq:g-contium}),
it follows that 
\begin{eqnarray*}
  & & g(\y^k + \lambda_{j-1}^k \Delta \y^k,
  \W^k + \lambda_{j-1}^k \Delta \W^k)
- g(\y^{k}, \W^{k}) \\
&=& \lambda_{j-1}^k
\nabla g (\y^k, \W^k) \bullet (\Delta \y^k, \Delta \W^k) \\
& & +
  \int_0^{\lambda_{j-1}^k}
  \left( 
\nabla g(\y^k + \lambda \Delta \y^k,
  \W^k + \lambda \Delta \W^k)
  - \nabla g(\y^k, \W^k) \right)
  \bullet (\Delta \y^k, \Delta \W^k) d\lambda \\
&\ge&  \lambda_{j-1}^k
\nabla g (\y^k, \W^k) \bullet (\Delta \y^k, \Delta \W^k)
- \frac{(\lambda_{j-1}^k)^2 L}{2} ||(\Delta \y^k, \Delta \W^k)||^2,
\end{eqnarray*}
since $\lambda^k_{j-1}\leq \overline{\lambda}^k$.
Combining these two inequalities, we obtain
$\lambda_{j-1}^k \ge \frac{2(1-\gamma)}{L}
\frac{\nabla g (\y^k, \W^k) \bullet (\Delta \y^k, \Delta \W^k)}
{||(\Delta \y^k, \Delta \W^k)||^2}$.
  It follows from (\ref{eq:inequality-g}) that 
\begin{eqnarray}
\frac{\nabla g (\y^k, \W^k) \bullet (\Delta \y^k, \Delta \W^k)}
{||(\Delta \y^k, \Delta \W^k)||^2} \ge \frac{1}{\alpha^k}
\ge \frac{1}{\alpha_{\max}}. \label{eq:min_nabla_g}
\end{eqnarray}
Since $\lambda_j^k$  is chosen from
$[\sigma_1 \lambda_{j-1}^k, \sigma_2 \lambda_{j-1}^k]$,
we obtain
$\lambda^k = \lambda_j^k \ge
\frac{2\sigma_1(1-\gamma)}{L \alpha_{\max}}$.
\end{proof}

We now   prove that 
the search direction generated by Algorithm~\ref{algo:dspg}
shrinks to zero in the infinite iterations.
\begin{LEMM}\label{lemm:liminf-delta}
Algorithm~\ref{algo:dspg} with $\epsilon = 0$ stops in a finite number of iterations attaining the
optimal value $g^*$, or
the infimum of the norm of the search direction tends to
zero as $k $ increases,
\begin{eqnarray*}
  \liminf_{k\to \infty} ||(\Delta \y^k_{(1)}, \Delta \W^k_{(1)})|| = 0.
  \end{eqnarray*}
\end{LEMM}

\begin{proof}
When Algorithm~\ref{algo:dspg} stops in a finite number of iterations, 
the optimality is guaranteed by Lemma~\ref{lemm:stop}.
From Lemma~\ref{lemm:direction1}, it is 
sufficient to prove
$\liminf_{k\to \infty} ||(\Delta \y^k, \Delta \W^k)|| = 0$,
Suppose, to contrary, that there exist $\delta > 0$ and an integer
$k_0$ such that $||(\Delta \y^k, \Delta \W^k)|| > \delta$
for $\forall k \ge k_0$.
Let us denote $g_k := g(\y^k, \W^k)$
and
$g_{\ell}^{\min} :=
\min\{g_{\ell M + 1}, \ldots, g_{(\ell+1) M}\}$.
It follows from Lemma~\ref{lemm:low-min}, 
(\ref{eq:g-condition}) and (\ref{eq:min_nabla_g}) that
\begin{eqnarray*}
g_{k+1} \ge
  \min\{g_{k}, \ldots, g_{k-M+1}\} +
  \widehat{\delta}
\quad \mbox{for} \quad \forall k \ge \max\{k_0,M\},
\end{eqnarray*}
where
$\widehat{\delta} = \gamma \min\{\overline{\lambda}_{\min},
\frac{2 \sigma_1 (1-\gamma)}{L \alpha_{\max}} \}
\frac{\delta^2}{\alpha_{\max}}$.

When $\ell$ is an integer such that $\ell > \frac{\max\{k_0,M\}}{M}$, 
we have
\begin{eqnarray*}
  g_{(\ell+1)M + 1}
  \ge \min\{
  g_{(\ell+1)M}, \ldots g_{(\ell+1)M - M + 1}
  \} + \widehat{\delta}
  = g_{\ell}^{\min} + \widehat{\delta}.
\end{eqnarray*}
By induction, for $j=2,\ldots,M$,
\begin{eqnarray*}
  g_{(\ell+1)M + j}
  \ge \min\{
  g_{(\ell+1)M + j-1}, \ldots g_{(\ell+1)M - M + j}
  \} + \widehat{\delta} 
  \ge \min\{
  g_{\ell}^{\min} + \widehat{\delta},
    g_{\ell}^{\min}\}  + \widehat{\delta}
  = g_{\ell}^{\min} + \widehat{\delta}.
\end{eqnarray*}
Therefore, we obtain
\begin{eqnarray*}
g_{\ell+1}^{\min} =
\min\{g_{(\ell+1) M + 1}, \ldots, g_{(\ell+1) M + M}\}  
\ge g_{\ell}^{\min} + \widehat{\delta}.
\end{eqnarray*}
From Lemma~\ref{lemm:level}, we know $ g(\y^0, \W^0) \le g_k \le g^*$ for each $k$.
Starting from an integer $\ell_0$ such that $\ell_0 > \frac{\max\{k_0,M\}}{M}$,
it follows that
\begin{eqnarray*}
  g^* \ge g_{\ell}^{\min} \ge g_{\ell_0}^{\min}
  + (\ell - \ell_0) \widehat{\delta}
  \ge g(\y^0, \W^0)
  + (\ell - \ell_0) \widehat{\delta}
  \quad \mbox{for} \quad \ell \ge \ell_0.
\end{eqnarray*}
When we take large $\ell$ such that $\ell > \ell_0 +
(g^* - g(\y^0, \W^0))/\widehat{\delta}$,
we have a contradiction. This completes the proof.

\end{proof}

 For the proof of the main theorem, we further investigate the behavior of the objective
function in Lemma~\ref{lemm:liminf-obj}, which requires 
Lemma \ref{lemm:short}. 
We use a matrix $\U^k \in \SMAT^n$ defined by
$U_{ij}^k := \rho_{ij} |X_{ij}^k| - W_{ij}^k X_{ij}^k$, 
and  
$\rho^{\max} := \max \{\rho_{ij} : i,j = 1,\ldots, n\}$.
The notation $[\Delta \W_{(1)}^k]_{ij}$ denotes the $(i,j)$th element of 
$\Delta \W_{(1)}^k = [\W^k + \X^k]_{\le \rho} - \W^k$.

\begin{LEMM}\label{lemm:short}
It holds that
\begin{eqnarray}
  |\U^k| \le \max\{2 \rho^{\max}, \eta_{\sX} \} |\Delta \W_{(1)}^k|.
  \label{eq:UV}
\end{eqnarray}

\end{LEMM}
\begin{proof}
We investigate the inequality by dividing into  three cases.
\begin{enumerate}
  \item Case $X_{ij}^k = 0$: We have $U_{ij}^k = 0$,
    hence (\ref{eq:UV}) holds.
  \item Case $X_{ij}^k > 0$: We have
    $U_{ij}^k = (\rho_{ij} - W_{ij}^k) X_{ij}^k \ge 0$.
    \begin{enumerate}
    \item Case $W_{ij}^k = \rho_{ij}$: We have
          $U_{ij}^k = 0$,
          hence (\ref{eq:UV}) holds.
    \item Case $W_{ij}^k < \rho_{ij}$:
          If $W_{ij}^k + X_{ij}^k \le \rho_{ij}$, then
          $[\Delta \W_{(1)}^k]_{ij} = W_{ij}^k + X_{ij}^k - W_{ij}^k = X_{ij}^k$. From $W_{ij}^k \ge -\rho_{ij}$,
          we have
          $0 \le U_{ij}^k = (\rho_{ij} - W_{ij}^k) [\Delta \W_{(1)}^k]_{ij}
          \le 2 \rho_{ij} [\Delta \W_{(1)}^k]_{ij} \le 2 \rho^{\max} |[\Delta \W_{(1)}^k]_{ij}|$.
          Otherwise, 
        if $W_{ij}^k + X_{ij}^k > \rho_{ij}$, then
        $[\Delta \W_{(1)}^k]_{ij} = \rho_{ij} - W_{ij}^k$, hence
          $U_{ij}^k = X_{ij}^k [\Delta \W_{(1)}^k]_{ij}$. From $|X_{ij}^k| \le ||\X^k|| \le \eta_{\sX}$,
          we obtain $0 \le U_{ij}^k \le \eta_{\sX} |[\Delta \W_{(1)}^k]_{ij}|$.
    \end{enumerate} 
  \item Case $X_{ij}^k < 0$: We compute simliarily  to 
    the case $X_{ij}^k > 0$.
\end{enumerate}
Combining these cases results in (\ref{eq:UV}).
\end{proof}

\begin{LEMM}\label{lemm:liminf-obj}
Algorithm~\ref{algo:dspg} with $\epsilon = 0$ stops
in a finite number of iterations attaining the
optimal value $g^*$, or
the  infimum of the difference of the objective functions between
$(\y^k, \W^k)$ and $(\y^*, \W^*) \in \FC^*$ tends to zero as $k$ increases, {\it i.e.},
\begin{eqnarray}
  \liminf_{k\to \infty} |g(\y^k,\W^k) - g^*| = 0.
\label{eq:liminf}
  \end{eqnarray}
\end{LEMM}
\begin{proof}
If Algorithm~\ref{algo:dspg} stops at the $k$th iteration,
$(\y^k, \W^k)$ is an optimal solution, therefore, $ g^* = g(\y^k, \W^k)$.
The proof for (\ref{eq:liminf})
is based on an inequality
\begin{eqnarray}
  |g(\y^k,\W^k) - g(\y^*, \W^*)|
  \le |g(\y^k,\W^k) - f(\X^k)| + |f(\X^k) - f(\X^*) | +
  |f(\X^*) - g(\y^*, \W^*)|. \label{eq:inequality-obj}
\end{eqnarray}
We know that $f(\X^*) = g(\y^*, \W^*)$ from the duality theorem, 
hence, we evaluate the first and second terms.

From the definition of $f$ and $g$, the first term will be bounded by
\begin{eqnarray}
& & |f(\X^k) - g(\y^k, \W^k)| \nonumber \\
&=& \left|\bold{\rho} \bullet |\X^k| - \W^k \bullet \X^k
  + (\AC(\X^k)-\b)^T \y^k \right| \nonumber \\
  &\le& \left|\bold{\rho} \bullet |\X^k| - \W^k \bullet \X^k \right|
+ \eta_{\y} ||\AC|| \cdot ||\X^k -\X^*||. \label{eq:fg}
\end{eqnarray}

Using Lemma~\ref{lemm:short}, we further have
\begin{eqnarray}
  \left|\bold{\rho} \bullet |\X^k| - \W^k \bullet \X^k \right|
= \left|\sum_{i=1}^n \sum_{j=1}^n U_{ij}^k\right| 
=\sum_{i=1}^n \sum_{j=1}^n |U_{ij}^k|
\le \max\{2 \rho^{\max}, \eta_{\sX}\} 
\sum_{i=1}^n \sum_{j=1}^n |[\Delta \W_{(1)}^k]_{ij}| \nonumber \\
\le \max\{2 \rho^{\max}, \eta_{\sX} \} n ||\Delta \W_{(1)}^k||
\le \max\{2 \rho^{\max}, \eta_{\sX} \} n
||(\Delta\y_{(1)}^k, \Delta\W_{(1)}^k)||. \label{eq:rhoW}
\end{eqnarray}
For the second inequality, we have used the relation between the two norms 
$\sum_{i=1}^n \sum_{j=1}^n |V_{ij}| \le  n ||\V||$ that holds
for any $\V \in \SMAT^n$. 

Next, we evaluate the second term of (\ref{eq:inequality-obj}).
Since $f_2(\X) = -\log \det (\X)$ is a convex function,
\begin{eqnarray*}
  f_2(\X^k) \ge f_2(\X^*) + \nabla f_2(\X^*)\bullet (\X^k -\X^*)
\end{eqnarray*}
and 
\begin{eqnarray*}
  f_2(\X^*) \ge f_2(\X^k) + \nabla f_2(\X^k)\bullet (\X^* -\X^k).
\end{eqnarray*}
These two inequalities indicate 
\begin{eqnarray*}
  |f_2(\X^k) - f_2(\X^*)|
  \le \max\{||\nabla f_2(\X^k)||, ||\nabla f_2(\X^*)||\} ||\X^k -\X^*||
  \le \eta_{\sX^{-1}} ||\X^k -\X^*||.
  \end{eqnarray*}
For the last inequality, we have used $\nabla f_2(\X) = -\X^{-1}$ for any $\X \succ \O$
and Remark~\ref{rema:x-bound}.
In addition, we have
\begin{eqnarray*}
  & & |\bold{\rho} \bullet (|\X^k| - |\X^*|)| 
  \le \sum_{i=1}^n \sum_{j=1}^n \rho_{ij}
  \left| | X_{ij}^k| - |X^*_{ij}| \right| \\
  &\le& \sum_{i=1}^n \sum_{j=1}^n \rho_{ij} | X_{ij}^k - X^*_{ij}|
  \le ||\bold{\rho}||\cdot || \X^k - \X^*||.
  \end{eqnarray*}
Hence, the second term of (\ref{eq:inequality-obj}) is bounded by
\begin{eqnarray}
& & |f(\X^k) - f(\X^*)| \nonumber \\
  &\le& | \C \bullet (\X^k - \X^*)|
  + \mu | f_2(\X^k) - f_2(\X^*)|
  + |\bold{\rho} \bullet (|\X^k| - |\X^*|)|  \nonumber \\
  &\le& (||\C|| + \mu \eta_{\X^{-1}} + ||\bold{\rho}||) ||\X^k - \X^*||. \label{eq:ff}
  \end{eqnarray}

We now evaluate the norm $||\X^k - \X^*||$.
It follows from  (P1) of Proposition~\ref{prop:hager} and $(\y^*, \W^*) \in \widehat{\WC}$ that
\begin{eqnarray*}
  \left( ( (\y^k,\W^k) + \nabla g(\y^k, \W^k) )
  - P_{\widehat{\sWC}}((\y^k,\W^k) + \nabla g(\y^k, \W^k))
  \right) \\
  \bullet 
  \left(  (\y^*,\W^*) 
  - P_{\widehat{\sWC}}((\y^k,\W^k) + \nabla g(\y^k, \W^k))
  \right) \le 0.
\end{eqnarray*}
Therefore, 
we obtain
\begin{eqnarray*}
  \left( \nabla g(\y^k, \W^k) ) -
  (\Delta \y_{(1)}^k, \Delta \W_{(1)}^k) \right)
  \bullet 
  \left(  (\y^*,\W^*) - (\y^k, \W^k)
  - (\Delta \y_{(1)}^k, \Delta \W_{(1)}^k) \right)
\le 0, 
\end{eqnarray*}
and this is equivalent to 
\begin{eqnarray}
& & \left(\Delta \y_{(1)}^k, \Delta \W_{(1)}^k\right)
  \bullet \left( (\y^*, \W^*) - (\y^k, \W^k) \right)
  + \nabla g(\y^k, \W^k)
  \bullet  \left(\Delta \y_{(1)}^k, \Delta \W_{(1)}^k\right)
- ||(\Delta \y_{(1)}^k, \Delta \W_{(1)}^k)||^2  \nonumber \\
&\ge& \nabla g(\y^k, \W^k) \bullet
  \left( (\y^*, \W^*) - (\y^k, \W^k)  \right). \label{eq:g-ineq}
\end{eqnarray}

On the other hand, it follows from (i) of Lemma~\ref{lemm:matrix-norm} that
\begin{eqnarray*}
  & & \left( \nabla g(\y^k, \W^k) -
  \nabla g(\y^*, \W^*)\right) \bullet
  \left((\y^*, \W^*) - (\y^k, \W^k)\right) \\
&=& (-\AC(\X^k) +\AC(\X^*), \X^k - \X^*) \bullet 
  (\y^* - \y^k, \W^* - \W^k) \\
&=& (\X^k - \X^*) \bullet (-\AC^T(\y^* - \y^k)) + (\X^k - \X^*) \bullet (\W^* - \W^k) \\
&=& (\X^k - \X^*) \bullet ((\C + \W^* - \AC^T(\y^*)) - (\C + \W^k - \AC^T(\y^k))) \\
  &=&  (\X^k - \X^*) \bullet (\mu(\X^*)^{-1} - \mu(\X^k)^{-1}) \\
  &\ge& \frac{\mu}{(\beta_{\sX}^{\max})^2} ||\X^k - \X^*||^2,
\end{eqnarray*}
and this is equivalent to
\begin{eqnarray*}
\nabla g(\y^k, \W^k) \bullet
  \left( (\y^*, \W^*) - (\y^k, \W^k)  \right) 
  &\ge& \frac{\mu}{(\beta_{\X}^{\max})^2} ||\X^k - \X^*||^2
  + \nabla g(\y^*, \W^*)
  \bullet
  \left( (\y^*, \W^*) - (\y^k, \W^k)  \right)
\end{eqnarray*}
By connecting this inequality and (\ref{eq:g-ineq}), we obtain
\begin{eqnarray*}
& & \left(\Delta \y_{(1)}^k, \Delta \W_{(1)}^k\right)
  \bullet \left( (\y^*, \W^*) - (\y^k, \W^k) \right)
  + \nabla g(\y^k, \W^k)
  \bullet  \left(\Delta \y_{(1)}^k, \Delta \W_{(1)}^k\right)
- ||(\Delta \y_{(1)}^k, \Delta \W_{(1)}^k)||^2  \nonumber \\
  &\ge& \frac{\mu}{(\beta_{\X}^{\max})^2} ||\X^k - \X^*||^2
  + \nabla g(\y^*, \W^*)
  \bullet
  \left( (\y^*, \W^*) - (\y^k, \W^k)  \right), 
\end{eqnarray*}
and this is equivalent to 
\begin{eqnarray}
& & \frac{\mu}{(\beta_{\sX}^{\max})^2} ||\X^k - \X^*||^2
- \left(\Delta \y_{(1)}^k, \Delta \W_{(1)}^k\right)
\bullet \left( (\y^*, \W^*) - (\y^k, \W^k) \right)
+ ||(\Delta \y_{(1)}^k, \Delta \W_{(1)}^k)||^2  \nonumber \\
& \le &  
  \nabla g(\y^k, \W^k)
  \bullet  \left(\Delta \y_{(1)}^k, \Delta \W_{(1)}^k\right)  
  - \nabla g(\y^*, \W^*)
  \bullet
  \left( (\y^*, \W^*) - (\y^k, \W^k)  \right). \label{eq:g-ineq2}
\end{eqnarray}

Since (\ref{eq:primal-dual}) and there is no duality gap,
we know that $\X^* \bullet \W^* = \bold{\rho} \bullet |\X^*|$.
Therefore,
\begin{eqnarray*}
& &\nabla g(\y^*, \W^*) \bullet
  \left((\y^*, \W^*) - (\y^k, \W^k) -
    (\Delta \y_{(1)}^k, \Delta \W_{(1)}^k) \right) \\
&=& (\b - \A(\X^*), \X^*)  \bullet
  \left((\y^*, \W^*) - (\y^k, \W^k) -
    (\Delta \y_{(1)}^k, \Delta \W_{(1)}^k) \right) \\
&=& (\0, \X^*) \bullet 
\left((\y^* - \y^k - \Delta \y_{(1)}^k, \W^* - \W^k -
    \Delta \W_{(1)}^k) \right) \\
  &=&  \X^* \bullet \W^* - \X^* \bullet [\W^k + \X^k]_{\le
  \bold{\rho}} \\
  &=&  |\X^*| \bullet \bold{\rho} - \X^* \bullet [\W^k + \X^k]_{\le
  \bold{\rho}} \\
  &\ge& 0.
\end{eqnarray*}
Hence, it follows that 
\begin{eqnarray}
  & & 
\nabla g(\y^k, \W^k) \bullet
  (\Delta \y_{(1)}^k, \Delta \W_{(1)}^k) 
-
\nabla g(\y^*, \W^*) \bullet
  \left( (\y^*, \W^*) - (\y^k, \W^k) \right)
    \nonumber \\
  & = & \left(\nabla g(\y^k, \W^k) - \nabla g(\y^*, \W^*) \right)
  \bullet (\Delta \y_{(1)}^k, \Delta \W_{(1)}^k) \nonumber \\
& & \quad  - \nabla g(\y^*, \W^*) \bullet
  \left( (\y^*, \W^*) - (\y^k, \W^k) - (\Delta \y_{(1)}^k, \Delta
    \W_{(1)}^k) \right) \nonumber \\
  & \le & \left(\nabla g(\y^k, \W^k) - \nabla g(\y^*, \W^*) \right)
  \bullet (\Delta \y_{(1)}^k, \Delta \W_{(1)}^k) \nonumber \\
&\le& ||\nabla g(\y^*, \W^*) - \nabla g(\y^k, \W^k) ||
  \cdot || (\Delta \y_{(1)}^k, \Delta \W_{(1)}^k)|| \nonumber \\
& = & ||(-\AC(\X^*)+\AC(\X^k), \X^* - \X^k)||
  \cdot || (\Delta \y_{(1)}^k, \Delta \W_{(1)}^k)|| \nonumber \\
& \le & (1 + ||\AC||)||\X^k - \X^*||
  \cdot || (\Delta \y_{(1)}^k, \Delta \W_{(1)}^k)||. \label{eq:g-ineq3}
\end{eqnarray}

From (\ref{eq:g-ineq2}) and (\ref{eq:g-ineq3}), we obtain
\begin{eqnarray*}
  & & \frac{\mu}{(\beta_{\sX}^{\max})^2} ||\X^k - \X^*||^2
  - (\Delta \y_{(1)}^k, \Delta \W_{(1)}^k) \bullet 
\left( (\y^*, \W^*) - (\y^k, \W^k) \right)
+ ||(\Delta \y_{(1)}^k, \Delta \W_{(1)}^k)||^2 \\
&\le& (1 + ||\AC||) ||\X^k - \X^*||
  \cdot || (\Delta \y_{(1)}^k, \Delta \W_{(1)}^k)||.
  \end{eqnarray*}
Using
$  || (\y^*, \W^*) - (\y^k, \W^k) ||
= \sqrt{||\y^* - \y^k||^2 + ||\W^* - \W^k||^2}
  \le  ||\y^* - \y^k|| + ||\W^* - \W^k||
  \le  ||\y^*|| + ||\y^k|| + ||\W^*|| + ||\W^k||
  \le 2 (\eta_{\y} + \eta_{\sW})$
and  $||(\Delta \y_{(1)}^k, \Delta \W_{(1)}^k)||^2 \ge 0$, it holds that 
\begin{eqnarray*}
  \frac{\mu}{(\beta_{\sX}^{\max})^2} ||\X^k - \X^*||^2
  - 2 (\eta_{\y} + \eta_{\sW}) ||(\Delta \y_{(1)}^k, \Delta \W_{(1)}^k)|| 
\le (1 + ||\AC||) ||\X^k - \X^*||
  \cdot || (\Delta \y_{(1)}^k, \Delta \W_{(1)}^k)||.
  \end{eqnarray*}
This is a quadratic inequality with respect to
$||\X^k - \X^*||$,
and solving this quadratic inequality leads us to
\begin{eqnarray}
  ||\X^k - \X^*|| \le u_1(||(\Delta \y_{(1)}^k, \Delta \W_{(1)}^k)||),
\label{eq:XX}
\end{eqnarray}
where
$u_1(t) :=
\frac{1}{2 \mu} (1+||\AC||)(\beta_{\sX}^{\max})^2 t+
  \frac{\beta^{\max}_{\sX}}{2 \mu} 
\sqrt{\left((1+||\AC||)(\beta_{\sX}^{\max})\right)^2 t^2
+ 8 \mu 
(\eta_{\y} + \eta_{\sW}) t }
$.

Using (\ref{eq:fg}), (\ref{eq:rhoW}), (\ref{eq:ff}) and (\ref{eq:XX}), the inequality (\ref{eq:inequality-obj}) is now
evaluated as
\begin{eqnarray*}
  |g(\y^k,\W^k) - g(\y^*, \W^*)| \le u_2(||(\Delta \y_{(1)}^k, \Delta \W_{(1)}^k)||)
\end{eqnarray*}
where
\begin{eqnarray}
u_2(t) := \max\{2\rho^{\max}, \eta_{\sX}\} n t  + 
(\eta_{\y} ||\AC|| + ||\C|| + \mu \eta_{\sX^{-1}}  + ||\rho||) u_1(t).
  \label{eq:u2}
\end{eqnarray}
Since all the coefficients are positive,
the function $u_2(t)$ is continuous for $t \ge 0$, and
$u_2(t) > 0$ for $t>0$.
Hence,  it follows Lemma~\ref{lemm:liminf-delta} that  
\begin{eqnarray*}
  \liminf_{k \to \infty} |g(\y^k,\W^k) - g(\y^*, \W^*)| = 0.
\end{eqnarray*}

\end{proof}

Finally, we are ready to show the main result, the convergence of the sequence
generated by Algorithm~\ref{algo:dspg} to the optimal value.

\begin{THEO}\label{theo:limit-obj}
Algorithm~\ref{algo:dspg} with $\epsilon = 0$ stops in a finite number of iterations attaining the
optimal value $g^*$, or generate a sequence $\{(\y^k, \W^k)\}$
such that
\begin{eqnarray*}
  \lim_{k\to \infty} |g(\y^k,\W^k) - g^*| = 0.
  \end{eqnarray*}
\end{THEO}

\begin{proof}
Suppose, to contrary, that there exists $\bar{\epsilon} > 0$ such that we have
an infinite sequence $\{k_1,k_2,\ldots,k_j,\ldots\}$ that satisfies
$g_{k_j}  < g^* -  \bar{\epsilon}$.

We should remark that it holds $k_{j+1} - k_j \le M$.
If $k_{j+1} - k_j > M$, since we can assume that
$g_{i} + \bar{\epsilon} \ge g^* $ for 
each $i \in [k_j + 1, \ldots, k_{j+1}-1]$,
the inequality (\ref{eq:g-condition})
indicates
$g_{k_{j+1}} \ge
\min\{g_{k_{j+1}-1}, \ldots, g_{k_{j+1}-M} \}
\ge g^* - \bar{\epsilon}$.
Hence, we know $k_{j+1} - k_j \le M$ and the sequence $\{k_1,k_2,\ldots,k_j,$ $\ldots\}$
should be actually infinite. 

Since $u_2(t)$ in (\ref{eq:u2}) is continuous for $t \ge 0$ and 
$u_2(t) > 0$ for $t > 0$, 
there exists $\bar{\delta}$ such that
$||(\Delta \y^{k_j}, \Delta \W^{k_j})|| > \bar{\delta} $ for each $j$.
We apply the same discussion as Lemma~\ref{lemm:liminf-delta} 
to the infinite sequence
$\{g_{k_1},g_{k_2},\ldots,g_{k_j},\ldots\}$.
 If $j$ becomes sufficiently large,
we have   a contradiction to the upper bound
$g_{k_j} \le g^*$.

\end{proof}

\section{Numerical Experiments}\label{sec:experiments}

We present numerical results obtained from implementing
Algorithm~\ref{algo:dspg}   on the randomly generated synthetic data,
 deterministic synthetic data and gene expression data in~\cite{LI10} which includes one of most
efficient computational results. 
Our numerical experiments were conducted on larger instances than the test problems  in~\cite{LI10} whenever it was
possible.

We compare our code DSPG, Algorithm~\ref{algo:dspg}, with the inexact primal-dual path-following
interior-point method (IIPM) \cite{LI10},
the Adaptive Spectral Projected Gradient method (ASPG) \cite{LU10}, and
the Adaptive
Nesterov's Smooth method (ANS) \cite{LU10}. 
For the gene expression data, our results are also compared with  the QUadratic approximation for sparse Inverse
Covariance estimation method (QUIC)~\cite{HSIEH14}
in Section~\ref{subsec:realdata}.
A comparison with the results on the Newton-CG primal proximal-point 
algorithm (PPA) \cite{WANG10} is not included
 since its performance was reported to be inferior to 
the IIPM \cite{LI10} and it failed to solve some instances.

We note that different stopping criteria are used in each of the aforementioned codes. 
They obviously affect the number of iterations and consequently the
overall computational time. 
For a fair comparison, we set
the threshold values for the IIPM, ASPG, ANS, and QUIC    
comparable to that of DSPG.
More precisely,
the stopping criteria of the DSPG was set to
\[ ||(\Delta \y^k_{(1)}, \Delta \W^k_{(1)})||_{\infty} \le \epsilon, \] 
where $\epsilon=10^{-5}$.
For the IIPM, we employed
\[
\max\left\{\frac{\textit{gap}}{1+|f(\X^k)|+|g(\y^k,\W^k)|},\textit{pinf},
\textit{dinf}\right\} \leq
\textrm{gaptol}:=10^{-6},
\]
where \textit{gap}, \textit{pinf}, \textit{dinf} were
specified in~\cite{LI10}, and  for the ASPG and ANS, we used two 
thresholds $\epsilon_0:=10^{-3}$ and 
$\epsilon_c:=10^{-5}$ such that $f(\X)\geq f(\X^*)-\epsilon_0$ and
$\max_{(i,j)\in\Omega}|X_{ij}|\leq \epsilon_c$ \cite{LI10}.
The QUIC stops when $\|\partial f(\X^k)\|/\textrm{Tr}{(\bold{\rho}}|\X^k|)< 10^{-6}$.

The DSPG was experimented  with the following parameters:
 $\gamma=10^{-4}$, $\tau=0.5$, $0.1=\sigma_1 < \sigma_2=0.9$, 
$\alpha_{\min}=10^{-15}=1/\alpha_{\max}$, $\alpha_0=1$, and $M=50$.
In the DSPG, the mexeig routine of the IIPM was used to reduce the computational time.
All numerical experiments were performed on a computer with Intel Xeon X5365 (3.0 GHz)
with 48 GB memory using MATLAB.

We set the initial solution as $(\y^0, \W^0) = (\0, \O)$, which satisfies the assumption (iii) for the instances tested 
in Sections~\ref{subsec:random} and~\ref{subsec:deterministic}.
Let $(\y^k,\W^k)$ be the output of Algorithm~\ref{algo:dspg}. The recovered
primal solution $\X^k:=\mu(\C+\W^k-\AC^T(\y^k))^{-1}$ may not 
satisfy the equalities $X_{ij}=0$ for $(i,j)\in\Omega$ in $(\PC)$ due to numerical errors. 
In this case,  we replace the value of ${X}_{ij}$ with $0$ for
$(i,j)\in\Omega$. 
For the tested instances, this replacement  
did not affect the semidefiniteness of $\X$, since the primal optimal solution 
was unique (Lemma~\ref{lemm:optset}) and the nonzero values of  
$X_{ij}$ were very small.

In the tables in Sections \ref{subsec:random} and \ref{subsec:deterministic}, the entry corresponding to the DSPG 
under the column ``primal obj." indicates the minimized 
function value $(\PC)$ for
$\X$  after replacing nonzero values  of ${X}_{ij}$ with $0$ for $(i,j)\in\Omega$, 
while ``gap" means the maximized function value $(\DC)$ 
for $(\y, \W)$ minus the primal one. Therefore, it should have a
minus sign.
The entries for the IIPM, ASPG, and ANS under  ``primal obj." column show the
difference between the corresponding function values and the 
 primal objective function values of the DSPG. Thus,
if this value is positive, it means that the DSPG obtained a lower
value for the minimization problem.
The tables also show the minimum eigenvalues for the primal variable,
number of (outer) iterations, and computational time.

In order to measure the effectiveness of recovering the inverse
covariance matrix $\bold{\Sigma}^{-1}$, 
we adopt the strategy in
\cite{LI10}. 
The normalized entropy loss (loss$_E$) and
the quadratic loss (loss$_Q$) are computed as
\[
\textrm{loss}_E:=\frac{1}{n}(\textrm{tr}(\bold{\Sigma}\X)
\log\det(\bold{\Sigma}\X)-n),
\qquad \textrm{loss}_Q:=\frac{1}{n}\|\bold{\Sigma}\X-\I\|, 
\]
respectively.
Notice that  the two values should ideally be zero if the regularity
term $\bold{\rho}\bullet|\X|$ is disregarded in $(\PC)$.
Also, the sensitivity and the specificity defined as
\[
\textrm{the sensitivity}:=\frac{\textrm{TP}}{\textrm{TP}+\textrm{FN}},
\qquad
\textrm{the specificity}:=\frac{\textrm{TN}}{\textrm{TN}+\textrm{FP}},
\]
are computed, 
where TP, TN, FP, and FN are the true positives, true negatives,
false positive, and false negative, respectively. 
In our case, the true positives are  correct nonzero entries in
$\bold{\Sigma}^{-1}$ and the true negatives are correct zero entries in
the same matrix. Therefore, the sensitivity and   specificity measure 
the correct rates of nonzero and of zero entries of $\bold{\Sigma}^{-1}$, respectively.
The values close to one for both sensitivity and specificity would be desirable.
Thus, we set values of $\rho>0$ such that $\bold{\rho}=\rho\E$ 
where $\E$ is the matrix of all ones in $(\PC)$ for which the
sensitivity and specificity become close to each other, and also $\mu$ equals to one.

\subsection{Randomly generated synthetic data}
\label{subsec:random}

As in \cite[Section~4.1]{LI10}, we  generated the test data by
first generating a sparse positive definite matrix $\bold{\Sigma}^{-1}\in\SMAT^n$ 
for a density parameter $\delta>0$, and then computing a sample covariance matrix $\C\in\SMAT^n$
 from $2n$ i.i.d.~random vectors selected from the $n$-dimensional
Gaussian distribution $\NC(\0,\bold{\Sigma})$.

Our experiments were carried out on different sizes $n$ of matrix $\bold{\Sigma}^{-1}$,  
two choices of density parameters $\delta=0.1$ and $0.9$, and 
problem $(\PC)$ without the linear
constraints $\AC(\X)=\b$ and with linear constraints $X_{ij}=0$ for 
$(i,j)\in\Omega$,
where $\Omega$ specifies the zero elements of  $\bold{\Sigma}^{-1}$.


\begin{table}[!htbp]
\footnotesize
\caption{Comparative numerical results for the DSPG, IIPM, ASPG and ANS on
unconstrained randomly generated synthetic data. 
$n=$1000, 3000, and 5000, density $\delta=0.1$ and $0.9$.}
\label{tab:randomunconstrained}
\begin{center}
\begin{tabular}{|c|c|c|r|r|r||c|r|} \hline
$n$ & $\rho$ & method & \multicolumn{1}{c|}{primal obj.} & \multicolumn{1}{c|}{iter.} & time (s) & \multicolumn{2}{c|}{$\delta=0.1$} \\ \hline
& & DSPG & $-$648.85805752 & 89 & 42.7 & $\lambda_{\min}(\X)$ &7.64$e-$02 \\ \cline{7-8}
& & (gap) & $-$0.00006098 & ~ & ~ & loss$_E$ & 1.8$e-$01 \\ \cline{7-8}
1000 & $5/1000$ & IIPM & $+$0.00000385 & 15 & 78.0 & loss$_Q$ & 2.2$e-$02 \\ \cline{7-8}
& = 0.005 & ASPG & $+$0.00046235 & 77 & 49.5 & sensitivity & 0.90 \\ \cline{7-8}
& & ANS & $+$0.00093895 & 310 & 172.4 & specificity & 0.88 \\ \hline \hline
& & DSPG & $-$4440.85648991 & 62 & 657.2 & $\lambda_{\min}(\X)$ & 2.42$e-$01  \\ \cline{7-8}
& & (gap) & $-$0.00009711 & ~ & ~ & loss$_E$ & 1.3$e-$02 \\ \cline{7-8}
3000 & $4/3000$ & IIPM & $+$0.00015710 & 15 & 1219.9 & loss$_Q$ & 2.0$e-$01 \\ \cline{7-8}
& = 0.001333 & ASPG & $+$0.00082640 & 49 & 801.9 & sensitivity & 0.82 \\ \cline{7-8}
& & ANS & $+$0.00089732 & 269 & 3255.9 & specificity & 0.85 \\ \hline \hline
& & DSPG & $-$9576.24150224 & 57 & 3015.4 & $\lambda_{\min}(\X)$ & 1.0$e-$02 \\ \cline{7-8}
& & (gap) & $-$0.00015026 & ~ & ~ & loss$_E$ & 1.9$e-$01 \\ \cline{7-8}
5000 & $3/5000$ & IIPM & $+$0.00039297 & 15 & 4730.0 & loss$_Q$ & 1.0$e-$02 \\ \cline{7-8}
& = 0.0006 & ASPG & $+$0.00012477 & 52 & 4137.0  & sensitivity & 0.82  \\ \cline{7-8}
& & ANS & $+$0.00084603 & 248 & 14929.4 & specificity & 0.81 \\ \hline \hline
%
$n$ & $\rho$ & method & \multicolumn{1}{c|}{primal obj.} & \multicolumn{1}{c|}{iter.} & time (s) & \multicolumn{2}{c|}{$\delta=0.9$}\\ \hline 
& & DSPG & $-$3584.93243464 & 33 & 16.5 & $\lambda_{\min}(\X)$ & 3.30$e+$01 \\ \cline{7-8}
& & (gap) & $-$0.00000122 & ~ & ~ & loss$_E$ & 9.4$e-$02 \\ \cline{7-8}
1000 & $0.15/1000$ & IIPM & $+$0.00031897 & 15 & 56.1 & loss$_Q$ & 1.5$e-$02 \\ \cline{7-8}
& = 0.00015 & ASPG & $+$0.00070753 & 21 & 18.0 & sensitivity & 0.50 \\ \cline{7-8} 
& & ANS & $+$0.00094435 & 78 & 49.2 & specificity & 0.53 \\ \hline \hline
& & DSPG & $-$13012.61749049 & 26 & 278.9 & $\lambda_{\min}(\X)$ & 7.56$e+$01 \\ \cline{7-8}
& & (gap) & $-$0.00000818 & ~ & ~ & loss$_E$ & 8.3$e-$02 \\ \cline{7-8}
3000 & $0.125/3000$ & IIPM & $+$0.00125846 & 18 & 1133.4 & loss$_Q$ & 8.1$e-$03 \\ \cline{7-8}
& = 0.0000417 & ASPG & $+$0.00049848 & 21 & 474.8 & sensitivity & 0.49 \\ \cline{7-8}
& & ANS & $+$0.00097430 & 81 & 1135.8 & specificity & 0.54 \\ \hline \hline
& & DSPG & $-$23487.45518427 & 26 & 1381.3 & $\lambda_{\min}(\X)$ & 1.07$e+$02 \\ \cline{7-8}
& & (gap) & $-$0.00000534 & ~ & ~ & loss$_E$ & 9.0$e-$02  \\ \cline{7-8}
5000 & $0.1/5000$ & IIPM & $+$0.00068521 & 23 & 5928.7 & loss$_Q$ & 6.5$e-$03  \\ \cline{7-8}
& = 0.00002 & ASPG & $+$0.00044082 & 21 & 2405.7 & sensitivity & 0.53 \\ \cline{7-8}
& & ANS & $+$0.00097990 & 90 & 6150.1 & specificity & 0.49 \\ \hline
\end{tabular}
\end{center}
\end{table}

Table~\ref{tab:randomunconstrained} shows the results for problems 
without any linear constraints in $(\PC)$.
Clearly, the DSPG requires less time to compute a lower objective 
value than the other codes. 
The advantage of the DSPG is greater for the denser 
problems ($\delta=0.9$, which is the case not considered in~\cite{LI10}) or larger
problems ($n=5000$).
Moreover, the dense problems tend to be easier to compute in terms of computational time, although their
recovery can be slightly worse than the problems with $\delta=0.1$,  as indicated by the values of the sensitivity and
specificity. For denser instances, loss$_E$ and loss$_Q$ are improved.

\begin{table}[!htbp]
\footnotesize
\caption{Comparative numerical results for the DSPG, 
IIPM, ASPG and ANS on
constrained randomly generated synthetic data.
$n=$1000, 3000, and 5000, density $\delta=0.1$ and $0.9$.}
\label{tab:randomconstrained}
\begin{center}
\begin{tabular}{|c|c|c|r|r|r||c|r|} \hline
$n$ & $\rho$/\# constraints & method & \multicolumn{1}{c|}{primal obj.} & \multicolumn{1}{c|}{iter.} & time (s) & \multicolumn{2}{c|}{$\delta=0.1$}\\ \hline
& $5/1000$ & DSPG & $-$631.25522377 & 144 & 76.1 & $\lambda_{\min}(\X)$ & 7.70$e-$02 \\ \cline{7-8}
& =0.005 & (gap) & $-$0.00013566 & ~ & ~ & loss$_E$ & 1.7$e-$01 \\ \cline{7-8}
1000 & & IIPM & $-$0.00013004 & 16 & 103.1 & loss$_Q$ & 2.1$e-$02 \\ \cline{7-8}
& 221,990 & ASPG & $+$0.00074651 & 1025 & 635.8 & sensitivity & 0.93 \\ \cline{7-8} 
& & ANS & $+$0.00076506 & 5464 & 3027.2 & specificity & 0.92 \\ \hline \hline
& $3/3000$ & DSPG & $-$4582.28297352 & 126 & 1383.8 & $\lambda_{\min}(\X)$ & 2.41$e-$01 \\ \cline{7-8}
& =0.001 & (gap) & $-$0.00006496 & ~ & ~ & loss$_E$ & 1.6$e-$01 \\ \cline{7-8}
3000 & & IIPM & $-$0.00004689 & 17 & 1692.4 & loss$_Q$ & 1.2$e-$02 \\ \cline{7-8}
& 1,898,796 & ASPG & $+$0.00062951 & 755 & 9658.4 & sensitivity & 0.92 \\ \cline{7-8}
& & ANS & $+$0.00083835 & 5863 & 67170.0 & specificity & 0.88 \\ \hline \hline
& $3/5000$ & DSPG & $-$9489.67203718 & 96 & 5180.8 & $\lambda_{\min}(\X)$ & 4.85$e-$01 \\ \cline{7-8}
& = 0.0006 & (gap) & $-$0.00005274 & ~ & ~ & loss$_E$ & 1.8$e-$01 \\ \cline{7-8}
5000 & & IIPM & $+$0.00001554 & 16 & 6359.0 & loss$_Q$ & 9.7$e-$03\\ \cline{7-8}
& 5,105,915 & ASPG & $+$0.00074531 & 704 & 43955.2 & sensitivity & 0.85 \\ \cline{7-8}
& & ANS & $+$0.00085980 & 5056 & 286746.6 & specificity & 0.89 \\ \hline
%
$n$ & $\rho$/\# constraints & method & \multicolumn{1}{c|}{primal obj.} & \multicolumn{1}{c|}{iter.} & time (s) & \multicolumn{2}{c|}{$\delta=0.9$}\\ \hline
& $0.1/1000$ & DSPG & $-$3625.96768067 & 42 & 20.7 & $\lambda_{\min}(\X)$ & 3.08$e+$01 \\ \cline{7-8}
& = 0.0001  & (gap) & $-$0.00000072 & ~ & ~ & loss$_E$ & 1.3$e-$01  \\ \cline{7-8}
1000 & & IIPM & $+$0.00014852 & 17 & 65.0 & loss$_Q$ & 1.9$e-$02 \\ \cline{7-8}
& 32,565 & ASPG & $+$0.00079319 & 376 &  547.9 & sensitivity & 0.64 \\ \cline{7-8}
& & ANS & $+$0.00098958 & 1938 & 1102.3 & specificity & 0.69 \\ \hline \hline
& $0.07/3000$ & DSPG & $-$13178.75746518 & 35 & 372.6 & $\lambda_{\min}(\X)$ & 6.84$e+$01  \\ \cline{7-8}
& = 0.0000233 & (gap) & $-$0.00000049 & ~ & ~ & loss$_E$ & 1.4$e-$01 \\ \cline{7-8}
3000 & & IIPM & $+$0.00089508 & 24 & 1528.1 & loss$_Q$ & 1.1$e-$02 \\ \cline{7-8}
& 238,977 & ASPG & $+$0.00030434 & 451 & 15295.4 & sensitivity & 0.67 \\ \cline{7-8} 
& & ANS & $+$0.00099513 & 3309 & 38990.8 & specificity & 0.67 \\ \hline \hline
& $0.07/5000$ & DSPG & $-$23644.31706813 & 29 & 1543.3 & $\lambda_{\min}(\X)$ & 1.01$e+$02 \\ \cline{7-8}
& = 0.000014 & (gap) & $-$0.00000833 & ~ & ~ & loss$_E$ & 1.2$e-$01 \\ \cline{7-8}
5000 & & IIPM & $+$0.00101943 & 28 & 7247.3 & loss$_Q$ & 7.9$e-$03 \\ \cline{7-8}
& 604,592 & ASPG & $+$0.00034229 & 344 & 30880.9 & sensitivity & 0.64-0.65 \\ \cline{7-8}
& & ANS & $+$0.00098642 & 3272 & 188957.0 & specificity & 0.68-0.69 \\ \hline
\end{tabular}
\end{center}
\end{table}

For the problems tested in Table~\ref{tab:randomconstrained}, the sparsity 
of $\bold{\Sigma}^{-1}\in\SMAT^n$ is imposed as
linear constraints in $(\PC)$ as $X_{ij}=0$ for $(i,j)\in\Omega$, where
$|\Omega|\equiv$``\# constraints''  in the table. 
From the results in Table~\ref{tab:randomconstrained}, we observe that the ASPG and ANS require much more 
computational time than in the 
unconstrained case.
The IIPM is the only code which violates the linear constraints 
$X_{ij}=0$ for $(i,j)\in\Omega$, resulting in values
less than $6.01\times 10^{-9}$ for $\max_{i,j=1,\ldots,n}|X_{ij}|$ at the final iteration.
We also see that loss$_E$ and loss$_Q$ do not change
when density $\delta$ is changed.

\subsection{Deterministic synthetic data}
\label{subsec:deterministic}

The numerical results on eight problems where
$\A\in\SMAT^n$
has a special structure such as diagonal band, fully dense,
or arrow-shaped {~\cite{LI10} are shown in Tables~\ref{tab:deterministicunconstrained} 
and~\ref{tab:deterministicconstrained}.
For each $\A$, a sample covariance matrix $\C\in\SMAT^n$ is computed
from $2n $ i.i.d.~random vectors selected from the
$n$-dimensional Gaussian distribution $\NC(\0,\A^{-1})$.
Finally, we randomly select 50\% of the zero entries for each
$\A$ to be the linear constraints in $(\PC)$, excepting for the
Full problem in Table~\ref{tab:deterministicunconstrained}.

Similar observation to Section \ref{subsec:random} can be made for the results presented in Tables~\ref{tab:deterministicunconstrained} and~\ref{tab:deterministicconstrained}.
The DSPG took less computational time  than the other methods in most cases and obtained slightly worse objective function values.

\begin{table}[!htbp]
\footnotesize
\caption{Comparative numerical results for the DSPG, IIPM, ASPG and ANS on
unconstrained deterministic synthetic data. $n=$2000.}
\label{tab:deterministicunconstrained}
\begin{center}
\begin{tabular}{|c|c|c|r|r|r||c|c|} \cline{1-6}
problem & $\rho$ & method & \multicolumn{1}{c|}{primal obj.} & \multicolumn{1}{c|}{iter.} & \multicolumn{1}{c|}{time (s)} & \multicolumn{2}{c}{}\\ \hline
&  & DSPG & 2189.07471338 & 20 & 57.8 & $\lambda_{\min}(\X)$ & 8.42$e-$01 \\ \cline{7-8}
& & (gap) & $-$0.33302912  & ~ & ~ & loss$_E$ & 7.9$e-$03 \\ \cline{7-8}
Full & 0.1 & IIPM & $-$0.33297893 & 11 & 185.9 & loss$_Q$ & 2.1$e-$03 \\ \cline{7-8}
& & ASPG & $-$0.33297903 & 54 & 244.1 & \multicolumn{2}{c}{} \\ 
& & ANS & $-$0.33283013 & 40 & 150.5 & \multicolumn{2}{c}{} \\ \cline{1-6}
\end{tabular}
\end{center}
\end{table}

\begin{table}[!htbp]
\footnotesize
\caption{Comparative numerical results for the DSPG, IIPM, ASPG and ANS on
constrained deterministic synthetic data. $n=$2000.}
\label{tab:deterministicconstrained}
\begin{center}
\begin{tabular}{|c|c|c|r|r|r||c|c|} \cline{1-6}
problem & $\rho$/\# constraints & method & \multicolumn{1}{c|}{primal obj.} & \multicolumn{1}{c|}{iter.} & \multicolumn{1}{c|}{time (s)} & \multicolumn{2}{c}{}\\ \hline
& 0.1 & DSPG & 3707.57716442 & 2001 & 6060.5 & $\lambda_{\min}(\X)$ & 1.00-1.25$e-$06 \\ \cline{7-8}
& & (gap) & $-$0.32561268 & ~ & ~ & loss$_E$ & 3.1$e-$02 \\ \cline{7-8}
ar1 & & IIPM & $-$0.32526710 & 38 & 3577.3 & loss$_Q$ & 2.3$e-$01 \\ \cline{7-8}
& 998,501 & ASPG & $-$0.32474270 &  19034 & 69534.6 & sensitivity & 1.00 \\ \cline{7-8}
& & ANS & $-$0.32448637 & 29347 & 88733.8 & specificity & 1.00 \\ \hline \hline
& 0.1 & DSPG & 3029.94934978 & 55 & 167.6 & $\lambda_{\min}(\X)$ & 2.73$e-$01 \\ \cline{7-8}
& & (gap) & $-$0.00329417 & ~ & ~ & loss$_E$ & 4.4$e-$02 \\ \cline{7-8}
ar2 & & IIPM & $-$0.00291044 & 11 & 290.1 & loss$_Q$ & 5.8$e-$03 \\ \cline{7-8}
& 997,502 & ASPG & $-$0.00309541 & 196 & 821.2 & sensitivity & 1.00 \\ \cline{7-8}
& & ANS & $-$0.00241116 & 1241 & 4230.7 & specificity & 1.00 \\ \hline \hline
& 0.03 & DSPG & 2552.71613399 & 78 & 236.8 & $\lambda_{\min}(\X)$ & 1.70$e-$01 \\ \cline{7-8}
& & (gap) & $-$0.00553547  & ~ & ~ & loss$_E$ & 1.8$e-$02 \\ \cline{7-8}
ar3 & & IIPM & $-$0.00545466 & 14 & 433.2 & loss$_Q$ & 4.4$e-$03 \\ \cline{7-8}
& 996,503 & ASPG & $-$0.00480321 & 353 & 1242.4 & sensitivity & 1.00 \\ \cline{7-8}
& & ANS & $-$0.00468946 & 2712 & 8592.6 & specificity & 1.00 \\ \hline \hline 
& 0.01 & DSPG & 2340.10866746 & 73 & 222.7 & $\lambda_{\min}(\X)$ & 2.31$e-$01 \\ \cline{7-8}
& & (gap) & $-$0.00050381  & ~ & ~ & loss$_E$ & 5.6$e-$02 \\ \cline{7-8}
ar4 & & IIPM & $-$0.00048223 & 14 & 403.3 & loss$_Q$ & 8.4$e-$03 \\ \cline{7-8}
& 995,505 & ASPG & $+$0.00030934 & 1095 & 3975.8 & sensitivity & 1.00 \\ \cline{7-8}
& & ANS & $+$0.00044155 & 5996 & 19379.6 & specificity & 1.00 \\ \hline \hline
& 0.1 & DSPG & 2253.67375651 & 14 &  44.4 & $\lambda_{\min}(\X)$ & 7.70$e-$01 \\ \cline{7-8}
& & (gap) & $-$0.00114736  & ~ & ~ & loss$_E$ & 1.5$e-$02 \\ \cline{7-8}
Decay & & IIPM & $-$0.00094913 & 10 & 170.5 & loss$_Q$ & 3.6$e-$03 \\ \cline{7-8}
& 981,586 & ASPG & $-$0.00106549 & 12 & 69.9 & sensitivity & 0.00 \\ \cline{7-8}
& & ANS & $-$0.00089883 & 32 & 126.6 & specificity & 1.00 \\ \hline \hline
& 0.1 & DSPG & 2204.50539735  & 82 & 248.7 & $\lambda_{\min}(\X)$ & 2.50-2.51$e-$07 \\ \cline{7-8}
& & (gap) & $-$0.00018704  & ~ & ~ & loss$_E$ & 4.8$e-$03 \\ \cline{7-8}
Star & & IIPM & $-$0.00001083 & 11 & 179.6 & loss$_Q$ & 4.5$e-$01 \\ \cline{7-8}
& 997,501& ASPG & $-$0.00002462 & 31 & 159.0 & sensitivity & 0.33 \\ \cline{7-8}
& & ANS & $-$0.00017677 & 92 & 311.4 & specificity & 1.00 \\ \hline \hline
& 0.05 & DSPG & 3519.14112855 & 1094 & 3307.0 & $\lambda_{\min}(\X)$ & 1.24-1.61$e-$06 \\ \cline{7-8}
& & (gap) & $-$0.07034481  & ~ & ~ & loss$_E$ & 2.9$e-$02 \\ \cline{7-8}
Circle & & IIPM & $-$0.07032168 & 28 & 1976.8 & loss$_Q$ & 2.6$e-$01 \\ \cline{7-8}
& 998,500 & ASPG & $-$0.06948986 & 11557 & 42437.1 & sensitivity & 1.00 \\ \cline{7-8}
& & ANS & $-$0.06946870 & 19714 & 59672.3 & specificity & 1.00 \\ \hline
\end{tabular}
\end{center}
\end{table}


\subsection{Gene expression data}
\label{subsec:realdata}

Five problems from the gene expression data~\cite{LI10} were tested for performance comparison.
Since it was assumed that  
the conditional independence of
their gene expressions is not known, linear constraints were not imposed in (${\mathcal P}$).
In this experiment,  we additionally compared the performance of the DSPG with 
QUIC~\cite{HSIEH14} which is known to be fast for sparse problems.

Figures~\ref{fig:lymph-er}-\ref{fig:hereditarybc} show the 
computational time (left axis) for each problem when 
$\rho$ is changed. As  $\rho$ grows larger, the final solution $\X^k$ (of the DSPG) becomes 
sparser, as shown in the right axis for the number of nonzero elements of $\X^k$.  

\begin{figure}[!htbp]
\begin{center}
\includegraphics[scale=0.5]{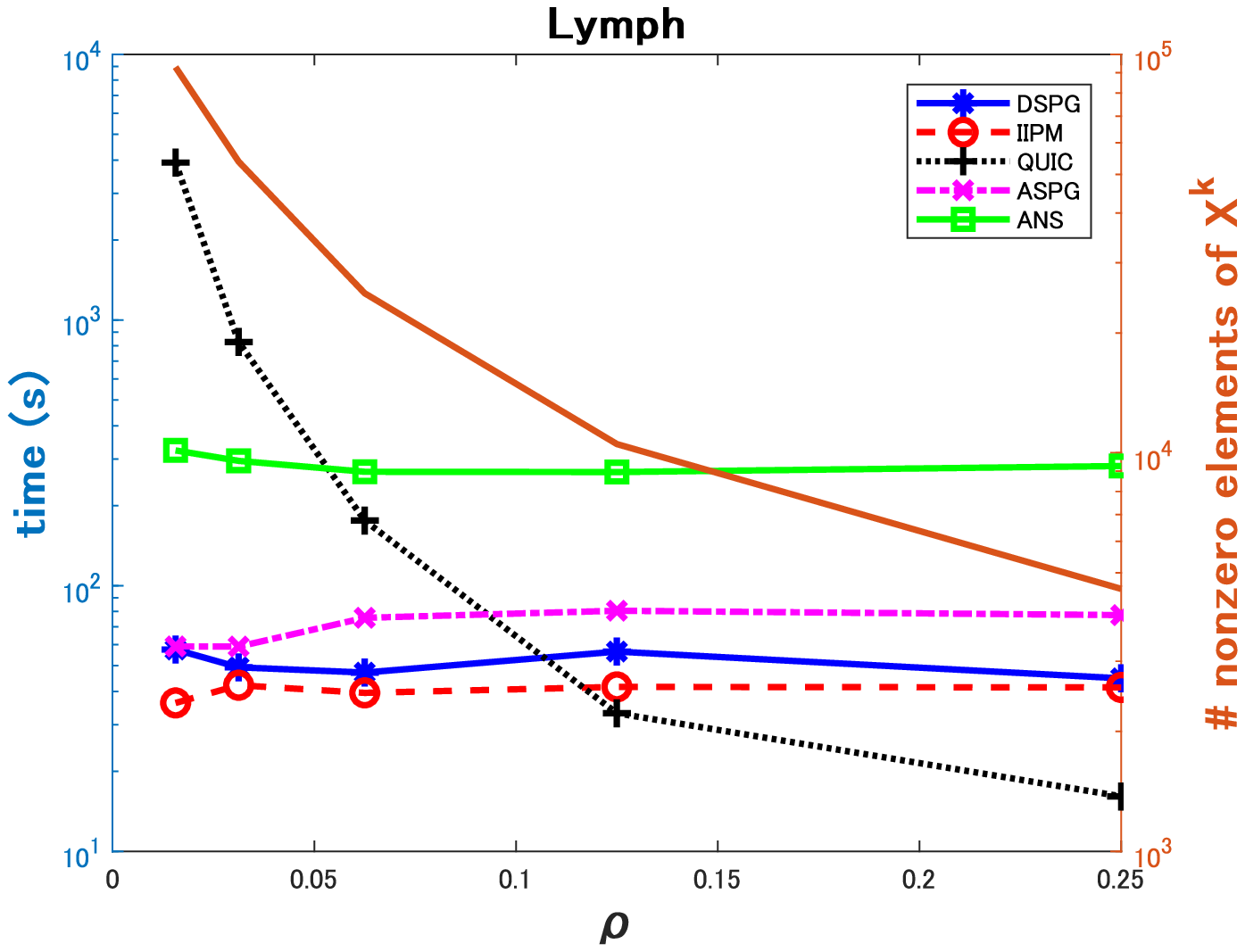}
\includegraphics[scale=0.5]{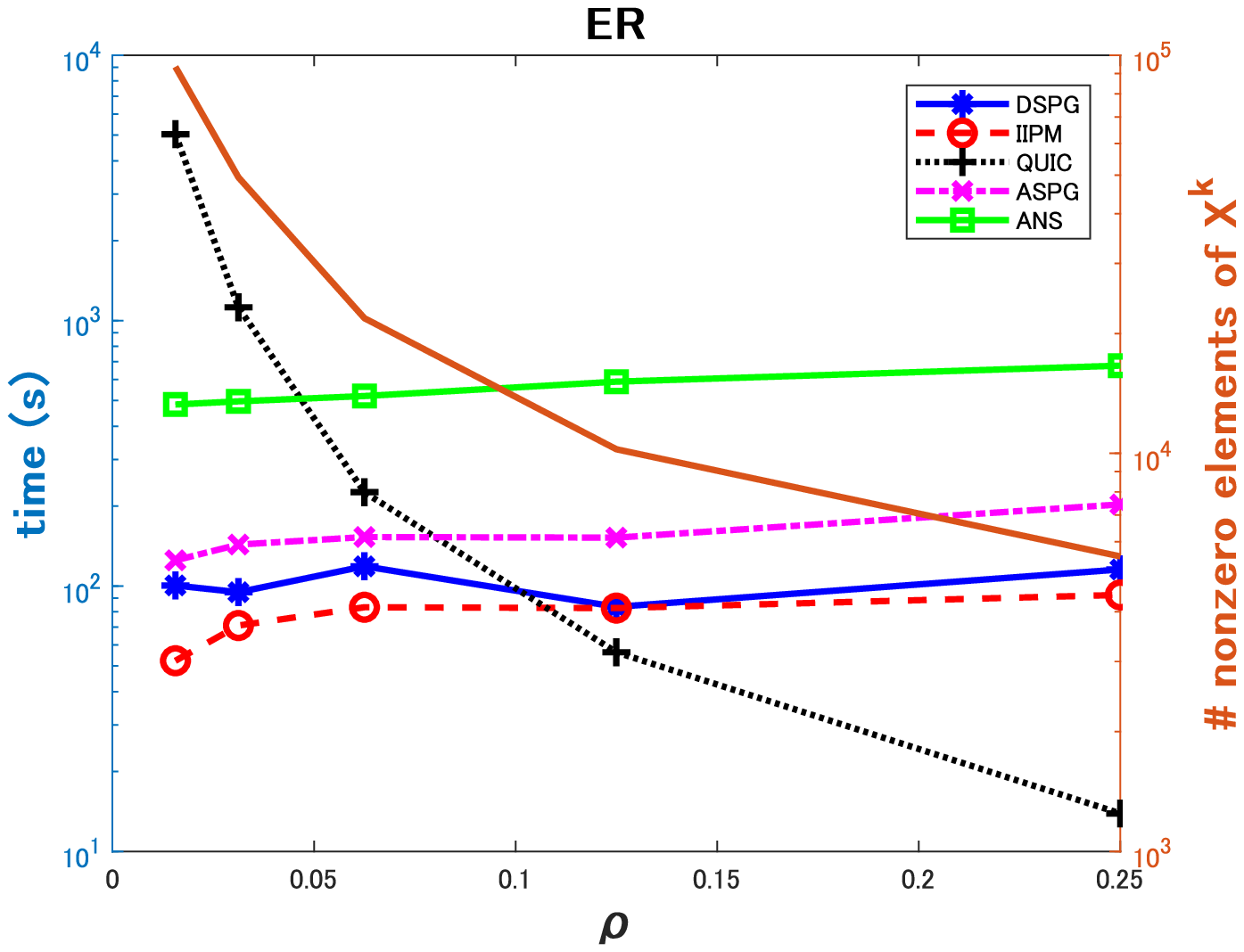}
\end{center}
\caption{Computational time (the left axis) for the DSPG, IIPM, QUIC, ASPG, ANS on the
problems ``Lymph" ($n=587$) and ``ER" ($n=692$) when $\rho$ is changed; the number of nonzero
elements of $\X^k$ for the final iterate of the DSPG (the right axis).}
\label{fig:lymph-er}
\end{figure}

\begin{figure}[!htbp]
\begin{center}
\includegraphics[scale=0.5]{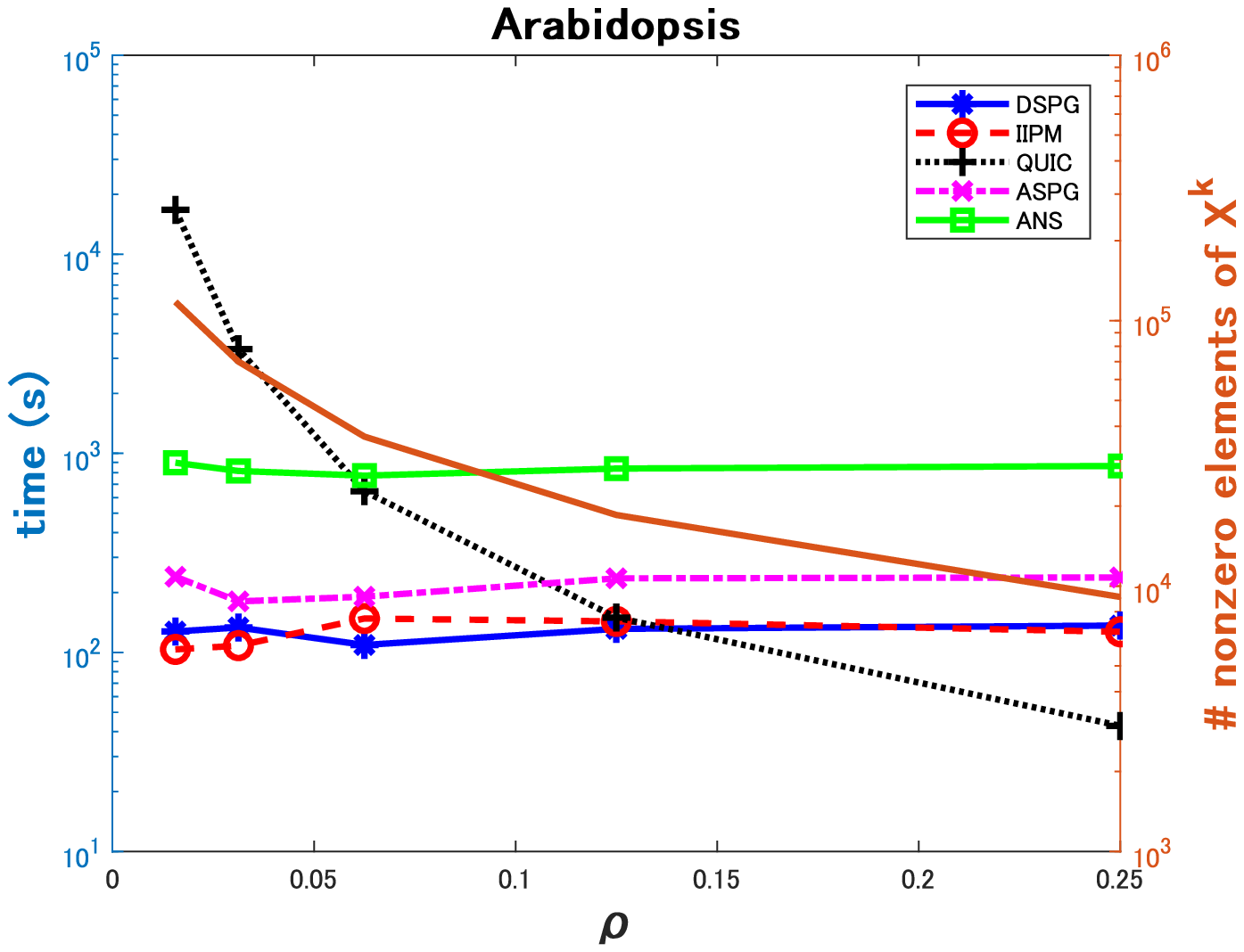}
\includegraphics[scale=0.5]{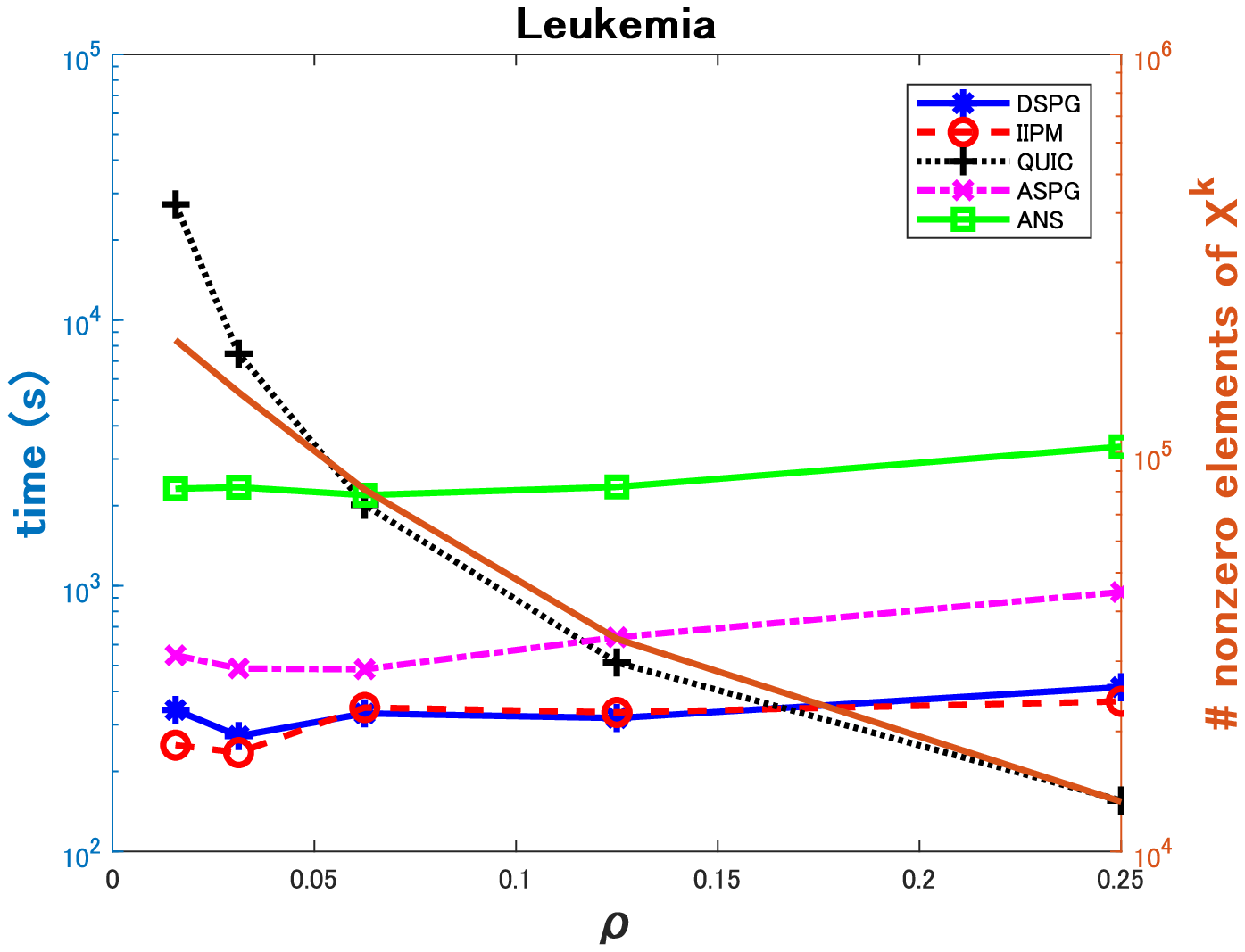}
\end{center}
\caption{Computational time (the left axis) for the DSPG, IIPM, QUIC, ASPG, ANS on
problems ``Arabidopsis'' ($n=834$) and ``Leukemia'' ($n=1255$) when 
$\rho$ is changed; the number of nonzero
elements of $\X^k$ for the final iteration of the DSPG (the right axis).}
\label{fig:arabidopsis-leukemia-er}
\end{figure}

\begin{figure}[!htbp]
\begin{center}
\includegraphics[scale=0.5]{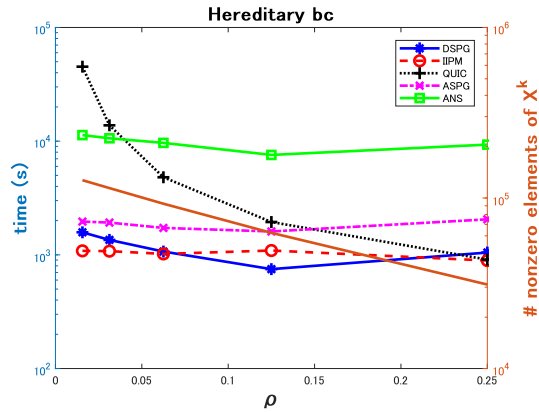}
\end{center}
\caption{Computational time (the left axis) for the DSPG, IIPM, QUIC, ASPG, ANS on the
problem ``Hereditary bc'' ($n=1869$) when $\rho$ is changed; the number of nonzero
elements of $\X^k$ for the final iteration of the DSPG (the right axis).}
\label{fig:hereditarybc}
\end{figure}

We see that the DSPG (solid blue line) is as 
competitive with the IIPM (dashed red line)
and even faster than the QUIC (dotted black line), which is known for their fast convergence, when $\rho$ is small.
The performance of the QUIC is closely related to the sparsity of the final
iterate of $\X^k$ for the DSPG (solid brown line) as expected. Here we used
the threshold $|\X^k|_{ij} \geq 0.05$ to determine nonzero elements.

\section{Conclusion}\label{sec:conclusion}

We have proposed a dual-type spectral projected gradient method for $(\PC)$ to efficiently handle large-scale problems.
Based on the theoretical convergence results of the proposed method, the Dual SPG algorithm has been implemented and
the numerical results on randomly generated synthetic data, deterministic synthetic data and gene expression data are reported. 
We have demonstrated the efficiency in computational time to obtain a better optimal value for $(\PC)$. In particular,
when $\rho$ is small, we have observed that the performance of the proposed method increases.

To further improve  the performance of the  Dual SPG method,  our future research includes reducing the computational
time by employing an approach similar to Dahl~\textit{et al.}~\cite{ZHANG18} and/or exploiting the structured sparsity
as discussed in \cite{SPARSECOLO}.

\bibliographystyle{plain} 
\begin{small}


\end{small}
\end{document}